\documentclass[a4paper, reqno]{amsart}
\usepackage{amsmath,mathrsfs}
\usepackage{amsfonts}
\usepackage{amssymb}
\usepackage{graphicx,color}
\usepackage[square,numbers]{natbib}
\newtheorem{theorem}{Theorem}

\newtheorem{lemma}[theorem]{Lemma}
\newtheorem{proposition}[theorem]{Proposition}
\newtheorem{algorithm}[theorem]{Algorithm}

\theoremstyle{definition}
\newtheorem{definition}[theorem]{Definition}
\newtheorem{assumptions}[theorem]{Assumptions}
\newtheorem{notation}[theorem]{Notation}

\theoremstyle{remark}
\newtheorem{remark}[theorem]{Remark}

\numberwithin{theorem}{section}
\numberwithin{equation}{section}

\def\XXint#1#2#3{{\setbox0=\hbox{$#1{#2#3}{\int}$}
     \vcenter{\hbox{$#2#3$}}\kern-.5\wd0}}

\newcommand{\dd}{\; \mathrm{d}}

\begin{document}
\title[A Universal Differentiability Set in The Heisenberg Group]{A Measure Zero Universal Differentiability Set in the Heisenberg Group}
\author[Andrea Pinamonti]{Andrea Pinamonti}
\address[Andrea Pinamonti]{Dipartimento di Matematica, Universita di Bologna, Piazza di Porta San Donato 5, Bologna 40126, Italy}
\email[Andrea Pinamonti]{Andrea.Pinamonti@gmail.com}
\author[Gareth Speight]{Gareth Speight}
\address[Gareth Speight]{Department of Mathematical Sciences, University of Cincinnati, 2815 Commons Way, Cincinnati, OH 45221, United States}
\email[Gareth Speight]{Gareth.Speight@uc.edu}

\begin{abstract}
We show that the Heisenberg group $\mathbb{H}^n$ contains a measure zero set $N$ such that every Lipschitz function $f\colon \mathbb{H}^n \to \mathbb{R}$ is Pansu differentiable at a point of $N$. The proof adapts the construction of small `universal differentiability sets' in the Euclidean setting: we find a point of $N$ and a horizontal direction where the directional derivative in horizontal directions is almost locally maximal, then deduce Pansu differentiability at such a point.
\end{abstract}

\maketitle

\section{Introduction}
Rademacher's theorem states that every Lipschitz function $f\colon \mathbb{R}^{n}\to \mathbb{R}^{m}$ is differentiable almost everywhere with respect to Lebesgue measure. This result is classical but has many applications and has inspired much research. One direction of this research is the extension of Rademacher's theorem to more general spaces, while another involves finding points of differentiability in extremely small sets. In this article we investigate both directions, by constructing a measure zero `universal differentiability set' in the Heisenberg group (Theorem \ref{maintheorem}), the most studied non-Euclidean Carnot group.

A Carnot group is a Lie group (smooth manifold which is also a group with smooth operations) whose Lie algebra (tangent space at the identity) admits a stratification. This stratification decomposes the Lie algebra as a direct sum of finitely many vector spaces; one of these consists of privileged `horizontal directions' which generate the other directions using Lie brackets. The stratification allows one to define dilations on the group. Carnot groups have a natural metric, defined using lengths of horizontal curves, and Haar measure invariant under group operations.

Using the group translations and dilations, one can define Pansu differentiability of functions between Carnot groups. Pansu's theorem states that Lipschitz functions between Carnot groups are Pansu differentiable almost everywhere with respect to the Haar measure \cite{Pan89}. This can be applied to show that every Carnot group (other than Euclidean space itself) contains no subset of positive measure which bi-Lipschitz embeds into a Euclidean space. Carnot groups are a source of many questions in analysis and geometry \cite{CDPT07, Gro96, LeD, Mon02, Spe14, Vit14}. 

The Heisenberg group $\mathbb{H}^{n}$ (Definition \ref{Heisenberg}) is the simplest non-Euclidean Carnot group, but is still not completely understood. For example, a geometric notion of intrinsic Lipschitz function between subgroups of $\mathbb{H}^{n}$ was introduced in \cite{FSC2} to study rectifiable sets \cite{FSSC2,FSSC3,Mag} and minimal surfaces \cite{CMPSC1,CMPSC2,Monti,SCV}. Intrinsic Lipschitz functions need not be metric Lipschitz but for certain subgroups they are intrinsically differentiable almost everywhere \cite{FSC2}. The most general statement for subgroups is not yet known.

Differentiability and Rademacher's theorem are also studied for functions between Banach spaces. There are versions of Rademacher's theorem for G\^{a}teaux differentiability of Lipschitz functions, but the case of the stronger Fr\'{e}chet differentiability is not fully understood \cite{LPT13}. One of the main ideas in the present article is that (almost local) maximality of directional derivatives implies differentiability. This was first used by Preiss \cite{Pre90} to find points of Fr\'{e}chet differentiability of Lipschitz functions on Banach spaces with separable dual. However, \cite{Pre90} does not give an `almost everywhere' type result. Indeed, it is not known if three real-valued Lipschitz functions on a separable Hilbert space have a common point of Fr\'{e}chet differentiability.

Cheeger \cite{Che99} gave a generalization of Rademacher's theorem for Lipschitz functions defined on metric spaces equipped with a doubling measure and satisfying a Poincar\'e inequality. This has inspired much research in the area of analysis on metric measure spaces. Bate \cite{Bat15} showed that Cheeger differentiability is strongly related to existence of many directional derivatives. However, in this context, the emphasis is on differentiability almost everywhere, rather than pointwise differentiability. 

A rather different direction of research asks whether one can find points of differentiability in extremely small sets. In particular, we can ask if Rademacher's theorem is sharp: given a set $N\subset \mathbb{R}^n$ of Lebesgue measure zero, does there exist a Lipschitz function $f\colon \mathbb{R}^{n} \to \mathbb{R}^{m}$ which is differentiable at no point of $N$? 

If $n\leq m$ the answer is yes: for $n=1$ this is rather easy \cite{Zah46}, while the general case is very difficult and combines ongoing work of multiple authors \cite{ACP10, CJ15}. This implies that in Rademacher's theorem, the Lebesgue measure cannot be replaced by a singular measure. The recent paper \cite{AM14} proves a Rademacher type theorem for an arbitrary finite measure, but in this case the directions of differentiability at almost every point depend on the measure.

If $n>m$ the answer to our question is no: there are Lebesgue measure zero sets $N\subset \mathbb{R}^n$ such that every Lipschitz function $f\colon \mathbb{R}^{n}\to \mathbb{R}^{m}$ is differentiable at a point of $N$. The case $m=1$ was a surprising corollary of the previously mentioned result in Banach spaces by Preiss \cite{Pre90}. The case $m>1$ were resolved by combining tools from the Banach space theory with a technique for avoiding porous sets \cite{PS15}. In all cases, maximizing directional derivatives had a crucial role.

Sets $N\subset \mathbb{R}^n$ containing a point of differentiability for every real-valued Lipschitz function are now called universal differentiability sets. The argument in \cite{Pre90} was greatly refined to show that $\mathbb{R}^{n}$, $n>1$, contains universal differentiability sets which are compact and of Hausdorff dimension one \cite{DM11, DM12}. This was improved to obtain a set which even has Minkowski dimension one \cite{DM14}. 

In the present article we show that one can adapt the ideas of \cite{Pre90} to the Heisenberg group. Our main result is Theorem \ref{maintheorem} which asserts the following: there is a Lebesgue measure zero set $N\subset \mathbb{H}^{n}$ such that every Lipschitz function $f\colon \mathbb{H}^{n} \to \mathbb{R}$ is Pansu differentiable at a point of $N$. This illustrates both the flexibility of Preiss' argument and the geometry of $\mathbb{H}^{n}$, which is rather far from being a Banach space. We now give an overview of the paper and information about the proof of Theorem \ref{maintheorem}. See Section \ref{prelim} for the relevant definitions.

In Section \ref{strongmaximal} we first introduce directional derivatives in horizontal directions for Lipschitz functions $f\colon \mathbb{H}^{n} \to \mathbb{R}$, and compare the supremum of directional derivatives with the Lipschitz constant. We then construct simple horizontal curves joining the origin to other points, and use these curves to study Pansu differentiability of the Carnot-Carath\'{e}odory distance. Finally we show that existence of a maximal horizontal directional derivative implies Pansu differentiability (Theorem \ref{maximalityimpliesdifferentiability}). This is an adaptation of a similar statement in Banach spaces \cite[Theorem 2.4]{Fit84}. We do not claim that such a maximal horizontal directional derivative exists; `almost maximal' horizontal directional derivatives have a role in the rest of the article.

In Section \ref{nullset} we define our `universal differentiability set' $N$, which may be chosen as any Lebesgue null $G_{\delta}$ set containing all horizontal lines joining points of $\mathbb{Q}^{2n+1}$. We construct useful horizontal curves inside this set, which allow us to modify a piece of a horizontal line to pass through a nearby point, without distorting the length or Lipschitz constant too much.

In Section \ref{sectionalmostmaxpansu} we first estimate how horizontal lines diverge in $\mathbb{H}^n$ and study some simple $\mathbb{H}$-linear maps. We then state Theorem \ref{almostmaximalityimpliesdifferentiability}, which is one of two theorems that taken together will prove Theorem \ref{maintheorem}. Theorem \ref{almostmaximalityimpliesdifferentiability} states that if $x\in N$ and $E$ is a horizontal direction at which the directional derivative $Ef(x)$ is `almost maximal', then $f$ is Pansu differentiable at $x$. This is an adaptation of \cite[Theorem 4.1]{Pre90}. Intuitively, almost maximality means that competing directional derivatives must satisfy an estimate which bounds changes in difference quotients by changes in directional derivatives. Such a bound is useful later when one wants to construct an almost maximal directional derivative. To prove Theorem \ref{almostmaximalityimpliesdifferentiability} one uses a contradiction argument. If $f$ is not differentiable at $x$ then there is some nearby point where the change in $f$ is too large. One modifies the line along which the directional derivative is large to form an auxillary curve inside $N$ which passes through the nearby point. On this curve we find a point and direction giving a larger directional derivative and satisfying the required bound on difference quotients. This gives a contradiction to almost maximality. New ideas in our adaptation to $\mathbb{H}^{n}$ include a restriction to horizontal directional derivatives and choosing an auxillary curve which is horizontal, with carefully estimated length and direction.

In Section \ref{construction} we adapt \cite[Theorem 3.1]{DM11} to show that one can actually construct an almost maximal directional derivative in the sense of Theorem \ref{almostmaximalityimpliesdifferentiability}. Proposition \ref{DoreMaleva} states that if $f_{0}\colon \mathbb{H}^{n}\to \mathbb{R}$ is Lipschitz then one can find $f\colon \mathbb{H}^{n} \to \mathbb{R}$ such that $f-f_{0}$ is $\mathbb{H}$-linear and $f$ has an almost maximal directional derivative at a point of $N$. This is done by constructing sequences $x_{n}\in N$, $E_{n}$ and $f_{n}$ such that $f_{n}-f_{0}$ is $\mathbb{H}$-linear and the directional derivatives $E_{n}f_{n}(x_{n})$ are closer and closer to being almost maximal compared to the allowed competitors. Then $x_{n} \to x_{\ast}$ and $f_{n} \to f$ for some $x_{\ast}$ and $f$ since at every stage the changes were small, while $E_{n} \to E$ for some $E$ since $\mathbb{H}$-linear perturbations are added to make directional derivatives in directions close to $E_{n}$ larger. We ensure $x_{\ast}\in N$ by using the fact that $N$ is $G_{\delta}$. Finally one must show that the directional derivative $Ef(x)$ exists and is almost maximal; to prove this it is crucial that at each stage we maximize over a constrainted set of points and directions. Our argument follows very closely that of \cite{DM11}, modified to use horizontal directions, $\mathbb{H}$-linear maps and H\"{o}lder equivalence of the Carnot-Carath\'{e}odory and Euclidean distance. Finally we observe that combining Theorem \ref{almostmaximalityimpliesdifferentiability} and Proposition \ref{DoreMaleva} gives Theorem \ref{maintheorem}.

One might ask if Theorem \ref{maintheorem} can be extended to general Carnot groups, or to construct compact sets of small dimension as in the Euclidean theory. Since this would require further new ideas and the proof of Theorem \ref{maintheorem} is already complicated, we do not address this here.

\medskip

\noindent \textbf{Acknowledgement.} This work was carried out with the support of the grant ERC ADG GeMeThNES. The authors thank Luigi Ambrosio for his support and David Preiss for suggesting that one first finds the analogue in $\mathbb{H}^{n}$ of the observation, previously applied in Banach spaces, that existence of a (truly, not almost) maximal directional derivative implies differentiability. 

The authors thank kind referees for suggesting numerous improvements to this article.

\section{Preliminaries}\label{prelim}
In this section we recall the Heisenberg group, horizontal curves, the Carnot-Carath\'{e}odory distance and Pansu differentiability \cite{BLU07, CDPT07, Gro96, LeD, Mon02, Vit14}.

Denote the Euclidean norm and inner product by $|\cdot|$ and $\langle \cdot, \cdot \rangle$ respectively. We represent points of $\mathbb{R}^{2n+1}$ as triples $(a,b,c)$, where $a, b\in \mathbb{R}^n$ and $c\in \mathbb{R}$.

\begin{definition}\label{Heisenberg}
The Heisenberg group $\mathbb{H}^n$ is $\mathbb{R}^{2n+1}$ equipped with the non commutative group law:
\[(a,b,c) (a',b',c') = (a+a', \, b+b', \, c+c'-2(\langle a, b'\rangle - \langle b, a'\rangle)).\]
\end{definition}

The identity element in $\mathbb{H}^n$ is $0$ and inverses are given by $x^{-1}=-x$.

\begin{definition}\label{Dilations}
For $r>0$ define dilations $\delta_{r}\colon \mathbb{H}^{n} \to \mathbb{H}^{n}$ by:
\[\delta_{r}(a,b,c)=(ra,rb,r^{2}c),\]
where $a, b\in \mathbb{R}^{n}$ and $c\in \mathbb{R}$.
\end{definition}

Dilations $\delta_{r}\colon \mathbb{H}^n \to \mathbb{H}^n$ and the projection $p\colon \mathbb{H}^n \to \mathbb{R}^{2n}$ onto the first $2n$ coordinates are group homomorphisms, where $\mathbb{R}^{2n}$ is considered as a group with the operation of addition. 

As sets there is no difference between $\mathbb{H}^{n}$ and $\mathbb{R}^{2n+1}$. Nevertheless, we sometimes think of elements of $\mathbb{H}^{n}$ as points and elements of $\mathbb{R}^{2n+1}$ as vectors. Let $e_{i}$ denote the standard basis vectors of $\mathbb{R}^{2n+1}$ for $1\leq i \leq 2n+1$. That is, $e_{i}$ has all coordinates equal to $0$ except for a $1$ in the $i$'th coordinate. Note that, in much of the literature on Carnot groups, the notation $\partial / \partial x_{i}$ is used instead of $e_{i}$. We next define a distinguished family of `horizontal' directions.

\begin{definition}\label{vectorfields}
For $1\leq i\leq n$ define vector fields on $\mathbb{H}^{n}$ by:
\[X_{i}(a,b,c)=e_{i}+2b_{i}e_{2n+1}, \quad Y_{i}(a,b,c)=e_{i+n}-2a_{i}e_{2n+1}.\]
Let $V=\mathrm{Span}\{X_{i}, Y_{i}: 1\leq i\leq n\}$ and $\omega$ be the inner product norm on $V$ making $\{X_{i}, Y_{i}: 1\leq i \leq n\}$ an orthonormal basis. We say that the elements of $V$ are horizontal vector fields or horizontal directions.
\end{definition}

An easy calculation shows that if $E\in V$ then
\begin{equation}\label{linestolines}x (tE(0)) = x+tE(x)\end{equation}
for any $x\in \mathbb{H}^n$ and $t\in \mathbb{R}$. That is, `horizontal lines' are preserved by group translations. If $E\in V$ then $E(0)$ is a vector $v\in \mathbb{R}^{2n+1}$ with $v_{2n+1}=0$. Conversely, for any such $v$ there exists $E\in V$ such that $E(0)=v$. If $E\in V$ then $p(E(x))$ is independent of $x$, so we can unambiguously define $p(E) \in \mathbb{R}^{2n}$. The norm $\omega$ is then equivalently given by $\omega(E)=|p(E)|$. 

We now use the horizontal directions to define horizontal curves and horizontal length in $\mathbb{H}^{n}$. Let $I$ denote a subinterval of $\mathbb{R}$. Recall that a map $\gamma\colon I \to \mathbb{R}^{2n+1}$ is absolutely continuous if it is differentiable almost everywhere and $\gamma(t)-\gamma(s)=\int_{s}^{t} \gamma'$ whenever $s, t \in I$.

\begin{definition}\label{horizontalcurve}
An absolutely continuous curve $\gamma\colon I \to \mathbb{H}^{n}$ is a horizontal curve if there is $h\colon I \to \mathbb{R}^{2n}$ such that for almost every $t\in I$:
\[\gamma'(t)=\sum_{i=1}^{n}(h_{i}(t)X_{i}(\gamma(t))+h_{n+i}(t)Y_{i}(\gamma(t))).\]
Define the horizontal length of such a curve by:
\[L_{\mathbb{H}}(\gamma)=\int_{I} |h|.\]
\end{definition}

Notice that in Definition \ref{horizontalcurve} we have $|(p \circ \gamma)'(t)|=|h(t)|$ for almost every $t$, so $L_{\mathbb{H}}(\gamma)$ is computed by integrating $|(p \circ \gamma)'(t)|$. That is, $L_{\mathbb{H}}(\gamma)=L_{\mathbb{E}}(p \circ \gamma)$, where $L_{\mathbb{E}}$ is the Euclidean length of a curve in Euclidean space. It can be shown that left group translations preserve horizontal lengths of horizontal curves.

In the next lemma we recall that horizontal curves are lifts of curves in $\mathbb{R}^{2n}$.

\begin{lemma}\label{lift}
An absolutely continuous curve $\gamma\colon I\to \mathbb{H}^{n}$ is a horizontal curve if and only if for almost every $t\in I$:
\[\gamma_{2n+1}'(t)=2\sum_{i=1}^{n} (\gamma_{i}'(t)\gamma_{n+i}(t)-\gamma_{n+i}'(t)\gamma_{i}(t)).\]
\end{lemma}

\begin{proof}
By definition, $\gamma$ is horizontal if and only if there exists $h \colon I \to \mathbb{R}^{2n}$ such that for almost every $t\in I$:
\[\gamma'(t)=\sum_{i=1}^n (h_{i}(t)X_{i}(\gamma(t))+h_{n+i}(t)Y_{i}(\gamma(t))).\]
Using Definition \ref{vectorfields}, the right hand side of this expression is exactly
\[(h_{1}(t), \ldots, h_{2n}(t), 2 \sum_{i=1}^{n} (h_{i}(t)\gamma_{n+i}(t)-2h_{n+i}(t)\gamma_{i}(t))).\]
By examining the initial coordinates we see $\gamma_{i}'(t)=h_{i}(t)$ for $1\leq i\leq 2n$. Hence:
\[\gamma_{2n+1}'(t)= 2\sum_{i=1}^{n} (\gamma_{i}'(t)\gamma_{n+i}(t)- \gamma_{n+i}'(t)\gamma_{i}(t))\]
for almost every $t\in I$.
\end{proof}

Any two points of $\mathbb{H}^{n}$ can be joined by a horizontal curve of finite horizontal length. This is a particular instance of Chow's Theorem in subriemannian geometry. We use this fact to define the Carnot-Carath\'{e}odory distance.

\begin{definition}\label{carnotdistance}
Define the Carnot-Carath\'{e}odory distance $d$ on $\mathbb{H}^{n}$ by:
\[d(x,y)=\inf \{L_{\mathbb{H}}(\gamma)\colon \gamma \mbox{ is a horizontal curve joining } x \mbox{ to }y\}.\]
Denote $d(x)=d(x,0)$ and $B_{\mathbb{H}}(x,r):=\{y\in \mathbb{H}^n \colon d(x,y)<r\}$.
\end{definition}

It is known that geodesics exist in the Heisenberg group. That is, the infimum in Definition \ref{carnotdistance} is actually a minimum. The Carnot-Carath\'{e}odory distance respects the group law and dilations \cite[(5.10) and (5.11)]{BLU07} - for every $g, x, y\in \mathbb{H}^{n}$ and $r>0$:
\begin{itemize}
\item $d(gx,gy)=d(x,y)$,
\item $d(\delta_{r}(x),\delta_{r}(y))=rd(x,y)$.
\end{itemize}

Notice $d(x,y)\geq |p(y)-p(x)|$, since the projection of a horizontal curve joining $x$ to $y$ is a curve in $\mathbb{R}^{2n}$ joining $p(x)$ to $p(y)$. Using the known fact that projections of geodesics in $\mathbb{H}^n$ are arcs of circles, it is also possible to show \cite[page 32]{CDPT07}:
\begin{equation}\label{squarerootinequality}
d(x)\geq \sqrt{|x_{2n+1}|}.
\end{equation}

The Carnot-Carath\'{e}odory distance and the Euclidean distance are topologically equivalent but not Lipschitz equivalent. However, they are H\"{o}lder equivalent on compact sets \cite[Corollary 5.2.10 and Proposition 5.15.1]{BLU07}.

\begin{proposition}\label{euclideanheisenberg}
Let $K\subset \mathbb{H}^{n}$ be a compact set. Then there exists a constant $C_{\mathrm{H}}=C_{\mathrm{H\ddot{o}lder}}(K) \geq 1$ such that for any $x, y\in K$:
\[|x-y|/C_{\mathrm{H}} \leq d(x,y)\leq C_{\mathrm{H}}|x-y|^{\frac{1}{2}}.\]
\end{proposition}

The Koranyi distance is Lipschitz equivalent to the Carnot-Carath\'{e}odory distance and given by the formula:
\begin{equation}\label{Koranyi}
d_{K}(x,y)=\|x^{-1}y\|_{K}, \mbox{ where } \|(a,b,c)\|_{K}=(|(a,b)|^{4}+c^{2})^{\frac{1}{4}}.
\end{equation}
We use the Koranyi distance occasionally to simplify some calculations.

We now make some simple observations about the distances introduced. They behave very simply if we follow a fixed horizontal direction.

\begin{lemma}\label{horizontaldistances}
If $E\in V$ then:
\begin{itemize}
\item $|E(0)|=\omega(E)=d(E(0))$,
\item $d(x,x+tE(x))=t\omega(E)$ for any $x\in \mathbb{H}^n$ and $t\in \mathbb{R}$.
\end{itemize}
\end{lemma}

\begin{proof}
The first equality is true because $\omega(E)=|p(E)|$, as noted earlier, and $|p(E)|=|E(0)|$, since $E(0)\in \mathbb{R}^{2n+1}$ has final coordinate equal to $0$.

The inequality $d(E(0))\leq \omega(E)$ is trivial since $t\mapsto tE(0)$, $t\in [0,1]$, is a horizontal curve joining $0$ to $E(0)$ of horizontal length exactly $\omega(E)$. 

Suppose $\gamma$ is a horizontal curve joining $0$ and $E(0)$. Then $L_{\mathbb{H}}(\gamma) = L_{\mathbb{E}} (p\circ \gamma)$ as noted in the discussion after Definition \ref{horizontalcurve}. Since $p\circ \gamma$ is a curve joining $0$ and $p(E)$ we deduce:
\[L_{\mathbb{E}}(p\circ \gamma)\geq |p(E)|=\omega(E).\]
This holds for any horizontal curve $\gamma$ joining $0$ to $E(0)$, so $d(E(0))\geq \omega(E)$. 

For the final statement we calculate as follows:
\begin{align*}
d(x,x+tE(x)) &= d(x,x(tE(0)))\\
&=d(tE(0))\\
&=t\omega(E).
\end{align*}
\end{proof}

If $f\colon \mathbb{H}^{n} \to \mathbb{R}$ or $\gamma\colon \mathbb{R} \to \mathbb{H}^{n}$ we denote the Lipschitz constant (not necessarily finite) of $f$ or $\gamma$ with respect to $d$ (in the domain or target respectively) by $\mathrm{Lip}_{\mathbb{H}}(f)$ and $\mathrm{Lip}_{\mathbb{H}}(\gamma)$. If we use the Euclidean distance then we use the notation $\mathrm{Lip}_{\mathbb{E}}(f)$ and $\mathrm{Lip}_{\mathbb{E}}(\gamma)$. Throughout this article `Lipschitz' means with respect to the Carnot-Carath\'{e}odory distance if the domain or target is $\mathbb{H}^n$, unless otherwise stated. For horizontal curves we have the following relation between Lipschitz constants.

\begin{lemma}\label{lipschitzhorizontal}
Suppose $\gamma \colon I \to \mathbb{H}^{n}$ is a horizontal curve. Then:
\[\mathrm{Lip}_{\mathbb{H}}(\gamma) =  \mathrm{Lip}_{\mathbb{E}}(p \circ \gamma).\]
\end{lemma}

\begin{proof}
Suppose $s, t\in I$ with $s<t$. Then:
\[|p(\gamma(t))-p(\gamma(s))| \leq d(\gamma(t),\gamma(s)) \leq \mathrm{Lip}_{\mathbb{H}}(\gamma)|t-s|.\]
Hence $\mathrm{Lip}_{\mathbb{E}}(p \circ \gamma) \leq \mathrm{Lip}_{\mathbb{H}}(\gamma)$. For the opposite inequality, notice $\gamma|_{[s,t]}$ is a horizontal curve joining $\gamma(s)$ to $\gamma(t)$. Since $|(p\circ \gamma)'|\leq \mathrm{Lip}_{\mathbb{E}}(p\circ \gamma)$, we can estimate as follows:
\begin{align*}
d(\gamma(s),\gamma(t)) &\leq L_{\mathbb{H}}(\gamma |_{[s,t]})\\
&= \int_{s}^{t} |(p\circ \gamma)'|\\
&\leq \mathrm{Lip}_{\mathbb{E}}(p\circ \gamma)(t-s).
\end{align*}
Hence $\mathrm{Lip}_{\mathbb{H}}(\gamma) \leq \mathrm{Lip}_{\mathbb{E}}(p \circ \gamma)$, which proves the lemma.
\end{proof}

Lebesgue measure $\mathcal{L}^{2n+1}$ is the natural Haar measure on $\mathbb{H}^{n}$. It is compatible with group translations and dilations \cite[page 44]{BLU07} - for every $g\in \mathbb{H}^{n}$, $r>0$ and $A\subset \mathbb{H}^{n}$:
\begin{itemize}
\item $\mathcal{L}^{2n+1}(\{gx\colon x\in A\})=\mathcal{L}^{2n+1}(A)$,
\item $\mathcal{L}^{2n+1}(\delta_{r}(A))=r^{2n+2}\mathcal{L}^{2n+1}(A)$.
\end{itemize}

We have defined group translations, dilations, distance and measure on $\mathbb{H}^n$. We can now introduce Pansu differentiability and state Pansu's theorem \cite{Pan89}. Intuitively, one replaces Euclidean objects by the corresponding ones in $\mathbb{H}^n$.

\begin{definition}\label{pansudifferentiability}
A function $L\colon \mathbb{H}^{n}\to \mathbb{R}$ is $\mathbb{H}$-linear if $L(xy)=L(x)+L(y)$ and $L(\delta_{r}(x))=rL(x)$ for all $x, y\in \mathbb{H}^{n}$ and $r>0$.

Let $f\colon \mathbb{H}^{n}\to \mathbb{R}$ and $x\in \mathbb{H}^{n}$. We say that $f$ is Pansu differentiable at $x$ if there is a $\mathbb{H}$-linear map $L \colon \mathbb{H}^{n}\to \mathbb{R}$ such that:
\[\lim_{y \to x} \frac{|f(y)-f(x)-L(x^{-1}y)|}{d(x,y)}=0.\]
In this case we say that $L$ is the Pansu derivative of $f$.
\end{definition}

Clearly a $\mathbb{H}$-linear map is Pansu differentiable at every point. Pansu's theorem is the natural version of Rademacher's theorem in $\mathbb{H}^{n}$.

\begin{theorem}[Pansu]
A Lipschitz function $f\colon \mathbb{H}^{n} \to \mathbb{R}$ is Pansu differentiable Lebesgue almost everywhere.
\end{theorem}

We can now state formally our main result.

\begin{theorem}\label{maintheorem}
There is a Lebesgue measure zero set $N\subset \mathbb{H}^n$ such that every Lipschitz function $f\colon \mathbb{H}^{n} \to \mathbb{R}$ is Pansu differentiable at a point of $N$.
\end{theorem}

\section{Maximality of directional derivatives implies Pansu differentiability}\label{strongmaximal}
We begin this section by defining directional derivatives in horizontal directions and comparing them to the Lipschitz constant (Lemma \ref{lipismaximal}). We then construct simple horizontal curves connecting the origin to other points and estimate their horizontal length and direction (Lemma \ref{goodcurve}). We use these curves to investigate Pansu differentiability of the Carnot-Carath\'{e}odory distance (Lemma \ref{differentiabilityofdistance}). Finally we prove that existence of a maximal horizontal directional derivative implies Pansu differentiability (Theorem \ref{maximalityimpliesdifferentiability}).

Proposition \ref{euclideanheisenberg} implies that curves in $\mathbb{H}^{n}$ which are Lipschitz with respect to the Carnot-Carath\'{e}odory distance are locally Lipschitz with respect to the Euclidean distance. Hence they are differentiable (in the usual sense) almost everywhere. We now show that derivatives of Lipschitz functions in horizontal directions can be defined by composing with any Lipschitz curve with tangent of the correct direction. Throughout this article $C$ will denote a constant that may change from line to line but remains bounded.
 
\begin{lemma}\label{welldefined}
Let $g, h\colon I \to \mathbb{H}^{n}$ be Lipschitz horizontal curves which are differentiable at $c\in I$ with $g(c)=h(c)$ and $g'(c)=h'(c)$. Let $f\colon \mathbb{H}^{n}\to \mathbb{R}$ be Lipschitz. If $(f\circ g)'(c)$ exists then $(f\circ h)'(c)$ exists and $(f\circ g)'(c)=(f\circ h)'(c)$.
\end{lemma}

\begin{proof}
For convenience assume that $c=0$ and $I$ is a neighbourhood of $0$. Using left group translations we may also assume that $g(0)=h(0)=0$. Since $f$ is Lipschitz, to prove the lemma it suffices to check that $d(g(t),h(t))/t\to 0$ as $t\to 0$. For this we use the Koranyi distance \eqref{Koranyi}. An easy calculation gives:
\begin{align*}
d_{K}(g(t),h(t)) & \leq C |h(t)-g(t)|\\
& \quad + C\Big| h_{2n+1}(t) - g_{2n+1}(t) + 2\sum_{i=1}^{n} (g_{i}(t)h_{n+i}(t) - g_{n+i}(t)h_{i}(t)) \Big|^{\frac{1}{2}}.
\end{align*}
Since $g(0)=h(0)=0$ and $g'(0)=h'(0)$ we see:
\begin{equation}\label{quotientstozero}
|h(t)-g(t)|/t \to |h'(0)-g'(0)|=0 \qquad \mbox{ as } t\to 0.
\end{equation}
Similarly for each $1\leq i \leq n$:
\[ (g_{i}(t)h_{n+i}(t) - g_{n+i}(t)h_{i}(t))/t^{2} \to g_{i}'(0)h_{n+i}'(0)-g_{n+i}'(0)h_{i}'(0)=0.\]
Hence it suffices to show that $(h_{2n+1}(t)-g_{2n+1}(t))/t^{2}\to 0$ as $t\to 0$. Since $h$ and $g$ are Lipschitz we use Lemma \ref{lipschitzhorizontal} to see that for $1\leq i\leq 2n$ and some constant $C$:
\[|h_{i}'(s)-g_{i}'(s)|\leq \mathrm{Lip}_{\mathbb{E}}(p\circ h - p \circ g)= \mathrm{Lip}_{\mathbb{H}}(h-g)\leq C.\]
Suppose without loss of generality that $t>0$. We use Lemma \ref{lift} to estimate as follows:
\begin{align*}
\frac{|h_{2n+1}(t)-g_{2n+1}(t)|}{t^2} & \leq \frac{2}{t^2}\sum_{i=1}^{n} \int_{0}^{t} |(h_{i}'-g_{i}')(h_{n+i}-g_{n+i}) - (h_{n+i}'-g_{n+i}')(h_{i}-g_{i})| \\
& \leq \frac{C}{t^2}\sum_{i=1}^{n} \int_{0}^{t} (|h_{n+i}-g_{n+i}|+|h_{i}-g_{i}|).
\end{align*}
Let $\varepsilon>0$. By \eqref{quotientstozero} we know $|h(s)-g(s)|\leq \varepsilon s$ for sufficiently small $s>0$. Hence for sufficiently small $t>0$:
\[ \frac{|h_{2n+1}(t)-g_{2n+1}(t)|}{t^2} \leq \frac{C}{t^2}\sum_{i=1}^{n} \int_{0}^{t} 2\varepsilon s \dd s \leq C \varepsilon,\]
for a new constant $C$ depending on $n$. This shows $(h_{2n+1}(t)-g_{2n+1}(t))/t^{2}\to 0$ as $t\to 0$, so concludes the proof.
\end{proof}

\begin{definition}\label{defdirectionalderivative}
Let $f\colon \mathbb{H}^{n}\to \mathbb{R}$ be a Lipschitz function and $E\in V$. Define $Ef(x):=(f\circ \gamma)'(t)$ whenever it exists, where $\gamma$ is any Lipschitz horizontal curve with $\gamma(t)=x$ and $\gamma'(t)=E(x)$.
\end{definition}

Notice that Lemma \ref{welldefined} implies Definition \ref{defdirectionalderivative} makes sense. Suppose $f$ is a Lipschitz function and $\gamma$ is a Lipschitz horizontal curve. Then $f\circ \gamma \colon \mathbb{R}\to \mathbb{R}$ is Lipschitz, so differentiable almost everywhere. Hence Lipschitz functions have many directional derivatives in horizontal directions. We often use horizontal lines $\gamma(t)=x+tE(x)$ to calculate directional derivatives when they exist:
\[Ef(x)=\lim_{t\to 0} \frac{f(x+tE(x))-f(x)}{t}.\]

Lipschitz constants with respect to the Euclidean distance are given by the supremum of directional derivatives over directions of Euclidean length $1$. We now prove an analogue of this for the Carnot-Carath\'{e}odory distance.

\begin{lemma}\label{lipismaximal}
Suppose $f\colon \mathbb{H}^n\to \mathbb{R}$ is Lipschitz. Then:
\[\mathrm{Lip}_{\mathbb{H}}(f)=\sup\{|Ef(x)| \colon x\in \mathbb{H}^{n}, \, E\in V, \, \omega(E)=1, \, Ef(x) \mbox{ exists}\}.\]
\end{lemma}

\begin{proof}
Temporarily denote the right side of the above equality by $\mathrm{Lip}_{\mathrm{D}}(f)$. Fix $x, y\in \mathbb{H}^n$ and a Lipschitz horizontal curve $\gamma \colon [0,d(x,y)]\to \mathbb{H}^{n}$ joining $x$ to $y$ such that $|(p\circ \gamma)'(t)|=1$ for almost every $t$. Let $G$ be the set of $t\in [0,d(x,y)]$ for which:
\begin{itemize}
\item $(f\circ \gamma)'(t)$ exists,
\item $\gamma'(t)$ exists,
\item $\gamma'(t)\in \mathrm{Span}\{X_{i}(\gamma(t)), Y_{i}(\gamma(t))\colon 1\leq i \leq n\}$,
\item $|(p\circ \gamma)'(t)|=1$.
\end{itemize}
Since $\gamma$ is a horizontal curve and $f\circ \gamma$ is Lipschitz, we know that $G$ has full measure. We estimate as follows:
\begin{align*}
|f(x)-f(y)|&= \Big|\int_{0}^{d(x,y)} (f\circ \gamma)' \Big|\\
&\leq d(x,y)\sup \{ |(f\circ \gamma)'(t)|\colon t\in G\}\\
&\leq d(x,y)\mathrm{Lip}_{\mathrm{D}}(f).
\end{align*}
Here we used Definition \ref{defdirectionalderivative}: $\gamma'(t)\in \mathrm{Span}\{X_{i}(\gamma(t)), Y_{i}(\gamma(t))\colon 1\leq i \leq n\}$ implies that there exists $E\in V$ with $E(\gamma(t))=\gamma'(t)$, and $|(p\circ \gamma)'(t)|=1$ then implies $\omega(E)=1$. Hence $\mathrm{Lip}_{\mathbb{H}}(f)\leq \mathrm{Lip}_{\mathrm{D}}(f)$.

For the opposite inequality fix $x\in \mathbb{H}^{n}$ and $E\in V$ such that $\omega(E)=1$ and $Ef(x)$ exists. Use Lemma \ref{horizontaldistances} to estimate as follows:
\begin{align*}
|Ef(x)|&= \Big| \lim_{t\to 0} \frac{f(x+tE(x))-f(x)}{t} \Big|\\
&\leq \limsup_{t \to 0} \frac{\mathrm{Lip}_{\mathbb{H}}(f)d(x,x+tE(x))}{t}\\
&=\mathrm{Lip}_{\mathbb{H}}(f).
\end{align*}
Hence $\mathrm{Lip}_{\mathrm{D}}(f) \leq \mathrm{Lip}_{\mathbb{H}}(f)$ which concludes the proof.
\end{proof}

The Carnot-Carath\'{e}odory distance $d$ is invariant under left group translations. Hence to understand $d$ it suffices to understand $d(x)=d(x,0)$ for $x\in \mathbb{H}^{n}$. For this purpose we construct explicit Lipschitz horizontal curves joining $0$ to points $x\in \mathbb{H}^{n}$. Our curves are simple concatenations of straight lines, but their Lipschitz constants and directions are sufficiently controlled for our applications.

\begin{lemma}\label{goodcurve}
Suppose $y\in \mathbb{H}^{n}$ with $p(y)\neq 0$. Write $y=(a,b,c)$ with $a, b\in \mathbb{R}^{n}$ and $c\in \mathbb{R}$. Denote $L=|p(y)|$ and define $\gamma\colon [0,1] \to \mathbb{H}^{n}$ by:
\[\gamma(t)= \begin{cases} t\Big(a-\frac{bc}{L^2},b+\frac{ac}{L^2},0\Big) &\mbox{if } 0\leq t\leq 1/2,\\
\frac{1}{2}\Big(a-\frac{bc}{L^2},b+\frac{ac}{L^2},0\Big)+\Big(t-\frac{1}{2}\Big)\Big(a+\frac{bc}{L^2},b-\frac{ac}{L^2},2c\Big) &\mbox{if }1/2 < t \leq 1. \end{cases}\]
Then:
\begin{enumerate}
\item $\gamma$ is a Lipschitz horizontal curve joining $(0,0,0)\in \mathbb{H}^n$ to $y=(a,b,c)\in \mathbb{H}^n$,
\item $\mathrm{Lip}_{\mathbb{H}}(\gamma) \leq L\Big(1+\frac{c^2}{L^4}+\frac{4c^2}{L^2}\Big)^{\frac{1}{2}}$,
\item $\gamma'(t)$ exists and $|\gamma'(t)-(a,b,0)|\leq \frac{c}{L}(1+4L^{2})^{\frac{1}{2}}$ for $t\in [0,1]\setminus \{1/2\}$ .
\end{enumerate}
We denote such a curve $\gamma$ by $\gamma_{y}$.
\end{lemma}

\begin{proof}
Notice $\gamma(0)=0$ and $\gamma(1)=(a,b,c)$. Recall that for $1\leq i\leq n$:
\[X_{i}(a',b',c')=e_{i}+2b_{i}'e_{2n+1}, \qquad Y_{i}(a',b',c')=e_{i+n}-2a_{i}'e_{2n+1}.\]
For $t\in (0,1/2)$ clearly
\[\gamma'(t)=\Big(a-\frac{bc}{L^2},b+\frac{ac}{L^2},0\Big)\]
and an easy calculation shows that
\[  \gamma'(t) = \sum_{i=1}^{n} \Big(a_{i}-\frac{b_{i}c}{L^2}\Big)X_{i}(\gamma(t))+\Big(b_{i}+\frac{a_{i}c}{L^2}\Big)Y_{i}(\gamma(t)).\]
For $t\in (1/2,1)$ we have:
\[\gamma'(t)=\Big(a+\frac{bc}{L^2},b-\frac{ac}{L^2},2c\Big).\]
We verify that for $t\in (1/2,1)$:
\[\gamma'(t)= \sum_{i=1}^{n} \Big(a_{i}+\frac{b_{i}c}{L^2}\Big)X_{i}(\gamma(t))+\Big(b_{i}-\frac{a_{i}c}{L^2}\Big)Y_{i}(\gamma(t)).\]
Validity of this equality for the first $2n$ coordinates is clear. The final coordinate of the right side is given by:
\begin{align*}
&\sum_{i=1}^{n} \Big(a_{i}+\frac{b_{i}c}{L^2}\Big)(2\gamma_{n+i}(t)) + \Big(b_{i}-\frac{a_{i}c}{L^2}\Big)(-2\gamma_{i}(t))\\
&\qquad =\sum_{i=1}^{n} \Big(a_{i}+\frac{b_{i}c}{L^2}\Big) \Big( \Big(b_{i}+\frac{a_{i}c}{L^2}\Big) + (2t-1)\Big( b_{i}-\frac{a_{i}c}{L^2} \Big) \Big)\\
&\qquad \quad+ \Big( b_{i}-\frac{a_{i}c}{L^2} \Big) \Big( -\Big( a_{i}-\frac{b_{i}c}{L^2} \Big) - (2t-1)\Big( a_{i}+\frac{b_{i}c}{L^2} \Big) \Big)\\
&\qquad=2\sum_{i=1}^{n} \Big(a_{i}+\frac{b_{i}c}{L^2}\Big) \frac{a_{i}c}{L^2} + \Big( b_{i}-\frac{a_{i}c}{L^2} \Big) \frac{b_{i}c}{L^2}\\
&\qquad=2\sum_{i=1}^{n} \Big(\frac{a_{i}^2c}{L^2} + \frac{b_{i}^{2}c}{L^2} \Big)\\
&\qquad=2c.
\end{align*}
Hence $\gamma$ is a horizontal curve. Since $\gamma$ is Lipschitz with respect to the Euclidean distance, Lemma \ref{lipschitzhorizontal} implies that $\gamma$ is a Lipschitz horizontal curve. This proves (1).

A straightforward computation shows that:
\[|\gamma'(t)|= \begin{cases} L\Big(1+\frac{c^2}{L^4} \Big)^{\frac{1}{2}} &\mbox{if } 0\leq t< 1/2,\\
L\Big(1+\frac{c^2}{L^4}+\frac{4c^2}{L^2}\Big)^{\frac{1}{2}} &\mbox{if }1/2 < t \leq 1. \end{cases}\]
Using Lemma \ref{lipschitzhorizontal} this gives the desired estimate of the Lipschitz constant:
\[\mathrm{Lip}_{\mathbb{H}}(\gamma)\leq \mathrm{Lip}_{\mathbb{E}}(\gamma)\leq L\Big(1+\frac{c^2}{L^4}+\frac{4c^2}{L^2}\Big)^{\frac{1}{2}}.\]
This proves (2). The estimate in (3) is also straightforward.
\end{proof}

Next we will study the Carnot-Carath\'{e}odory distance near points of the form $u=E(0)$ for some $E\in V$.

\begin{lemma}\label{differentiabilityofdistance}
Fix $u_{1}, u_{2} \in \mathbb{R}^n$ not both zero and let $u=(u_{1},u_{2},0) \in \mathbb{H}^{n}$. Then:
\begin{enumerate}
\item $d(uz) \geq d(u)+ \langle z, u/d(u)\rangle$ for any $z\in \mathbb{H}^{n}$,
\item $d(uz)=d(u)+\langle z, u/d(u)\rangle+ o(d(z))$ as $z\to 0$. That is, the Pansu derivative of $d$ at u is $L(x):=\langle x , u/d(u) \rangle$.
\end{enumerate}
\end{lemma}

\begin{proof}
We may assume that $d(u)=1$ throughout the proof, since the general statement can be deduced using dilations. To prove (1) first recall $d(x)\geq |p(x)|$ for all $x\in \mathbb{H}^n$, while Lemma \ref{horizontaldistances} shows $d(u)=|p(u)|$ for our particular choice of $u$. Clearly also $\langle p(z),p(u)\rangle = \langle z,u\rangle$ for such $u$. We use Pythagoras' theorem and $d(u)=|p(u)|=1$ to estimate as follows:
\begin{align*}
d(uz) &\geq |p(uz)|\\
&= |p(u)+p(z)|\\
&= | p(u)(1+\langle p(z), p(u) \rangle ) + (p(z)-\langle p(z), p(u) \rangle p(u) )|\\
&\geq |p(u)| (1+\langle p(z),p(u) \rangle )\\
&= |p(u)|+\langle p(z), p(u)\rangle\\
&= d(u)+\langle z, u\rangle.
\end{align*}

To prove (2) it suffices to show that $d(uz)\leq d(u)+\langle z, u/d(u)\rangle+ o(d(z))$ as $z\to 0$. Let $(a,b,c)=uz$ and $L=|p(uz)|=|p(u)+p(z)|$. Assume $d(z)\leq 1/2$. Using $|p(u)|=1$ and $|p(z)|\leq d(z)\leq 1/2$, we see $1/2\leq L\leq 2$. Using the formula for the group law and $|z_{2n+1}|\leq d(z)^2$ from \eqref{squarerootinequality}, we can estimate $c$ as follows:
\begin{align*}
|c|&\leq |z_{2n+1}| + 4|p(u)||p(z)|\\
&\leq d(z)^2 + 4d(z)\\
&\leq 5d(z).
\end{align*}
Lemma \ref{goodcurve} and the definition of the Carnot-Carath\'{e}odory distance gives:
\begin{align*}
d(zw) &\leq L\Big(1+\frac{c^2}{L^4}+\frac{4c^2}{L^2}\Big)^{\frac{1}{2}}\\
&\leq L(1+800d(z)^2)^{\frac{1}{2}}\\
& \leq L+o(d(z)).
\end{align*}
To conclude the proof of (2) we claim that $L\leq 1+\langle z,u\rangle + o(d(z))$. Estimate as follows:
\begin{align*}
L&=|p(u)+p(z)|\\
&=|p(u)(1+\langle p(z),p(u)\rangle) + (p(z)-\langle p(z), p(u)\rangle p(u)) |\\
&=( (1+\langle p(z),p(u) \rangle)^{2} + |p(z)-\langle p(z),p(u) \rangle p(u)|^{2} )^{\frac{1}{2}}\\
&\leq ( (1+\langle p(z),p(u) \rangle)^{2} + 4d(z)^{2} )^{\frac{1}{2}}\\
&=(1+\langle p(z),p(u) \rangle)\Big(1+\frac{4d(z)^{2}}{(1+\langle p(z),p(u) \rangle)^{2}}\Big)^{\frac{1}{2}}\\
&\leq (1+\langle z,u \rangle)\Big(1+\frac{2d(z)^{2}}{(1+\langle z,u \rangle)^{2}}\Big).
\end{align*}
The claim then follows since $d(z)/(1+\langle z, u \rangle)\leq 2d(z)\to 0$ as $d(z)\to 0$.
\end{proof}

We now use Lemma \ref{differentiabilityofdistance} to show that existence of a maximal horizontal directional derivative implies Pansu differentiability. This is an adaptation of \cite[Theorem 2.4]{Fit84}. By Lemma \ref{lipismaximal}, existence of a maximal horizontal directional derivative is equivalent to the agreement of a directional derivative with the Lipschitz constant.

\begin{theorem}\label{maximalityimpliesdifferentiability}
Let $f\colon \mathbb{H}^n \to \mathbb{R}$ be Lipschitz, $x\in \mathbb{H}^n$ and $E\in V$ with $\omega(E)=1$. Suppose $Ef(x)$ exists and $Ef(x)=\mathrm{Lip}_{\mathbb{H}}(f)$. Then $f$ is Pansu differentiable at $x$ with derivative $L(x):=\mathrm{Lip}_{\mathbb{H}}(f)\langle x, E(0) \rangle = \mathrm{Lip}_{\mathbb{H}}(f)\langle p(x), p(E) \rangle$.
\end{theorem}

\begin{proof}
Let $0<\varepsilon\leq 1/2$. Using Lemma \ref{differentiabilityofdistance}, we can choose $0<\alpha \leq \varepsilon$ such that whenever $d(z) \leq \alpha$:
\[ d(E(0)z) - d(E(0)) \leq \langle z, E(0) \rangle + \varepsilon d(z).\]
Use existence of $Ef(x)$ to fix $\delta>0$ such that whenever $|t|\leq \delta$:
\[|f(x+tE(x))-f(x)-tEf(x)|\leq \alpha^{2}|t|.\]

Suppose that $0<d(w) \leq \alpha \delta$ and $t=\alpha^{-1}d(w)$. Then $0<t\leq \delta$, $d(\delta_{t^{-1}}(w))=\alpha$ and $2d(w)=2\alpha t\leq t$. Recall that $\omega(E)=1$ implies $d(E(0))=1$. We use also left invariance of the Carnot-Carath\'{e}odory distance to estimate as follows:
\begin{align*}
f(xw)-f(x) &= (f(xw) - f(x-tE(x))) + (f(x-tE(x)) - f(x))\\
&\leq \mathrm{Lip}_{\mathbb{H}}(f)d(xw,x -tE(x)) - tEf(x) + \alpha^{2}t\\
&= \mathrm{Lip}_{\mathbb{H}}(f)d(xw,x(-tE(0))) - tEf(x) + \alpha^{2}t\\
&= \mathrm{Lip}_{\mathbb{H}}(f)d((tE(0))w) - t\mathrm{Lip}_{\mathbb{H}}(f) + \alpha^{2}t\\
&= t \mathrm{Lip}_{\mathbb{H}}(f)( d(E(0)\delta_{t^{-1}}(w)) - d(E(0))) + \alpha^{2}t\\
&\leq t \mathrm{Lip}_{\mathbb{H}}(f)( \langle \delta_{t^{-1}}(w), E(0) \rangle + \varepsilon d(\delta_{t^{-1}}(w))) + \alpha^{2}t\\
&\leq \mathrm{Lip}_{\mathbb{H}}(f)\langle w, E(0) \rangle + \varepsilon \mathrm{Lip}_{\mathbb{H}}(f) d(w) + \alpha d(w)\\
&\leq \mathrm{Lip}_{\mathbb{H}}(f)\langle w, E(0) \rangle + \varepsilon (\mathrm{Lip}_{\mathbb{H}}(f) +1)d(w).
\end{align*}
For the opposite inequality we have:
\begin{align*}
f(xw)-f(x) &= (f(xw)-f(x+tE(x))) + (f(x+tE(x))-f(x))\\
&\geq -\mathrm{Lip}_{\mathbb{H}}(f)d(xw,x+tE(x)) + tEf(x) - \alpha^{2}t\\
&= -\mathrm{Lip}_{\mathbb{H}}(f)d(xw,x(tE(0))) + tEf(x) - \alpha^{2}t\\
&= -\mathrm{Lip}_{\mathbb{H}}(f) d((-tE(0))w)+ t\mathrm{Lip}_{\mathbb{H}}(f) - \alpha^{2}t\\
&= - t\mathrm{Lip}_{\mathbb{H}}(f)(d((-E(0))\delta_{t^{-1}}(w)) - d(E(0))) - \alpha^{2}t\\
&\geq -t\mathrm{Lip}_{\mathbb{H}}(f)(\langle \delta_{t^{-1}}(w), -E(0) \rangle + \varepsilon d(\delta_{t^{-1}}(w))) - \alpha^{2}t\\
&\geq -\mathrm{Lip}_{\mathbb{H}}(f)\langle w, -E(0) \rangle - \varepsilon \mathrm{Lip}_{\mathbb{H}}(f)d(w) - \alpha^{2}t\\
&\geq \mathrm{Lip}_{\mathbb{H}}(f)\langle w, E(0) \rangle - \varepsilon\mathrm{Lip}_{\mathbb{H}}(f)d(w)-\alpha d(w)\\
&=\mathrm{Lip}_{\mathbb{H}}(f)\langle w, E(0)\rangle - \varepsilon(\mathrm{Lip}_{\mathbb{H}}(f)+1)d(w).
\end{align*}
This shows that $d(w) \leq \alpha\delta$ implies:
\[|f(xw)-f(x)-\mathrm{Lip}_{\mathbb{H}}(f)\langle w, E(0) \rangle |\leq (\mathrm{Lip}_{\mathbb{H}}(f)+1)\varepsilon d(w).\]
Hence $f$ is Pansu differentiable at $x$ with derivative $\mathrm{Lip}_{\mathbb{H}}(f)\langle \cdot,E(0) \rangle$.
\end{proof}

An arbitrary Lipschitz function may not have a maximal horizontal directional derivative as in Theorem \ref{maximalityimpliesdifferentiability}. Construction of almost locally maximal horizontal directional derivatives inside a measure zero set, and deducing Pansu differentiability at such points, is the content of the rest of the paper.

\section{The universal differentiability set and horizontal curves}\label{nullset}
In this section we identify our measure zero universal differentiability set (Lemma \ref{uds}) and construct useful horizontal curves inside this set (Lemma \ref{newcurveg}). Recall that a set in a topological space is $G_{\delta}$ if it is a countable intersection of open sets.

\begin{lemma}\label{uds}
There is a Lebesgue measure zero $G_{\delta}$ set $N\subset \mathbb{H}^{n}$ containing all straight lines which are also horizontal curves and join pairs of points of $\mathbb{Q}^{2n+1}$. Any such set contains the image of:
\begin{enumerate}
\item the line $x+tE(x)$ whenever $x\in \mathbb{Q}^{2n+1}$ and $E\in V$ is a linear combination of $\{X_{i}, Y_{i}: 1\leq i\leq n\}$ with rational coefficients,
\item all curves of the form $x\gamma_{y}$ for $x, y\in \mathbb{Q}^{2n+1}$ with $p(y)\neq 0$, where $\gamma_{y}$ is the curve constructed in Lemma \ref{goodcurve}.
\end{enumerate}
\end{lemma}

\begin{proof}
There exist open sets of of arbitrarily small measure containing a fixed line. By taking a union of countably many such open sets with decreasing measure, we find an open set of arbitrarily small measure containing all horizontal lines joining points of $\mathbb{Q}^{2n+1}$. By taking a countable intersection of such sets with measures converging to $0$, we obtain the required measure zero $G_{\delta}$ set $N$.

The horizontal line $x+tE(x)$ joins $x$ to $x+E(x)$ so its image is a subset of $N$ whenever $x\in \mathbb{Q}^{2n+1}$ and $E\in V$ is a linear combination of $\{X_{i}, Y_{i}: 1\leq i\leq n\}$ with rational coefficients. By examining the formula for $\gamma_{y}$ in Lemma \ref{goodcurve}, we see each curve $x \gamma_{y}$, for $x, y\in \mathbb{Q}^n$ with $p(y)\neq 0$, is a union of two such horizontal lines. This proves the lemma.
\end{proof}

To prove Theorem \ref{almostmaximalityimpliesdifferentiability} (an almost maximal directional derivative implies Pansu differentiability) we will modify horizontal line segments (along which a Lipschitz function will have a large directional derivative) to pass through nearby points (which intuitively show non-Pansu differentiability at some point). In the next lemma we see how to do this without changing the length or direction of the line too much.

\begin{lemma}\label{newcurveg}
Given $\eta>0$, there is $0<\Delta(\eta)<1/2$ and $C_{\mathrm{m}}=C_{\mathrm{modify}}\geq 1$ such that the following holds whenever $0<\Delta<\Delta(\eta)$. Suppose:
\begin{itemize}
\item $x, u\in \mathbb{H}^n$ with $d(u)\leq 1$,
\item $E\in V$ with $\omega(E)=1$,
\item $0<r<\Delta$ and $s:=r/\Delta$.
\end{itemize}
Then there is a Lipschitz horizontal curve $g \colon \mathbb{R} \to \mathbb{H}^n$ such that:
\begin{enumerate}
\item $g(t)=x+tE(x)$ for $|t|\geq s$,
\item $g(\zeta)=x\delta_{r}(u)$, where $\zeta:=r\langle u,E(0)\rangle$,
\item $\mathrm{Lip}_{\mathbb{H}}(g)\leq 1+\eta \Delta$,
\item $g'(t)$ exists and $|(p\circ g)'(t)- p(E)| \leq C_{\mathrm{m}}\Delta$ for $t\in \mathbb{R}$ outside a finite set.
\end{enumerate}
Suppose additionally $x, u\in \mathbb{Q}^{2n+1}$, $E$ is a linear combination of $\{X_{i}, Y_{i}: 1\leq i\leq n\}$ with rational coefficients and $r, s \in \mathbb{Q}$. Then $g$ is a concatenation of curves from Lemma \ref{uds}(1,2).
\end{lemma}

\begin{proof}
The distance $d$ is invariant under left group translations and the group law is linear in the first $2n$ coordinates. Hence to prove (1)--(4) we can assume $x=0$. From the proof it will be clear that this also suffices for the final statement.

For $|t|\geq s$ the curve $g(t)$ is explicitly defined by (1) and satisfies (3) and (4). To define $g(t)$ for $|t|<s$ we consider the two cases $-s< t\leq \zeta$ and $\zeta \leq t < s$. These are similar so we show how to define the curve for $-s< t\leq \zeta$. By using left group translations by $\pm sE(0)$ and reparameterizing the curve, it suffices to show the following claim.

\medskip

\emph{Claim.} Suppose $0< \Delta< \eta/1632$. Then there exists a Lipschitz horizontal curve $\varphi \colon [0,s+\zeta] \to \mathbb{H}^n$ such that $\varphi(0)=0$, $\varphi(s+\zeta)=(sE(0))\delta_{r}(u)$ and:
\begin{enumerate}
\item[(A)] $\mathrm{Lip}_{\mathbb{H}}(\varphi)\leq 1+\eta \Delta$,
\item[(B)] $\varphi'(t)$ exists and $|(p\circ \varphi)'(t)-p(E)| \leq 184 \Delta$ for $t\in [0,s+\zeta]$ outside a finite set.
\end{enumerate}

\medskip

\emph{Proof of Claim.}
Let $(a,b,c):=(sE(0))\delta_{r}(u)$ and $L:=|(a,b)|$. Observe:
\[L=|sp(E) + rp(u)|=s|p(E)+\Delta p(u)|.\]
Our assumptions imply $|p(E)|=\omega(E)=1$ and $|p(u)|\leq d(u)\leq 1$. Using also \eqref{squarerootinequality} gives $|u_{2n+1}|\leq 1$. Since $0<\Delta< 1/2$ we deduce $s/2 \leq L \leq 2s$. Definition \ref{Heisenberg} and Definition \ref{Dilations} give:
\[|c|\leq r^2 |u_{2n+1}| + 4rs\leq r^2+4rs\leq 5rs.\]
Lemma \ref{goodcurve} provides a Lipschitz horizontal curve $\gamma \colon [0,1] \to \mathbb{H}^{n}$ joining $0$ to $(a,b,c)$ such that:
\begin{enumerate}
\item[(A')] $\mathrm{Lip}_{\mathbb{H}}(\gamma) \leq L\Big( 1 + \frac{c^2}{L^4} + \frac{4c^2}{L^2} \Big)^{\frac{1}{2}}$,
\item[(B')] $\gamma'(t)$ exists and $|\gamma'(t)-(a,b,0)|\leq \frac{c}{L} (1+4L^{2})^{\frac{1}{2}}$ for $t\in [0,1]\setminus \{1/2\}$.
\end{enumerate}
We verify the claim with $\varphi\colon [0,s+\zeta]\to \mathbb{H}^{n}$ defined by $\varphi(t)=\gamma(t/(s+\zeta))$. Notice that $\varphi$ is a Lipschitz horizontal curve with $\varphi(0)=0$ and $\varphi(s+\zeta)=(sE(0))\delta_{r}(u)$. 

\medskip

\emph{Proof of (A).} We first develop the estimate (A'). For this we use our estimates of $c$ and $L$, the inequality $s<1$, and the equality $r=\Delta s$:
\begin{align*}
\mathrm{Lip}_{\mathbb{H}}(\gamma) &\leq L\Big( 1 + \frac{c^2}{L^4} + \frac{4c^2}{L^2} \Big)^{\frac{1}{2}}\\
&\leq L\Big( 1+400(r^2/s^2)+400r^2\Big)^{\frac{1}{2}}\\
&\leq L\Big( 1+800\Delta^2\Big)^{\frac{1}{2}}\\
&\leq L + 800s\Delta^2.
\end{align*}
To estimate $L=|p( (sE) (\delta_{r}(u))  )|$ more carefully, first recall:
\[\zeta=r\langle u,E(0)\rangle= r\langle p(u),p(E)\rangle.\]
We use the orthogonal decomposition:
\begin{equation}\label{orthogonaldecomposition}
p( (sE) \delta_{r}(u) )=(s+\zeta)p(E) + (rp(u)-\zeta p(E)).
\end{equation}
Using $d(u)\leq 1$ and $\omega(E)\leq 1$ gives $|\zeta|\leq r\leq s/2$ and $|rp(u)-\zeta p(E)|\leq 2r$. We estimate as follows:
\begin{align*}
L&= ( (s+\zeta)^2 + |rp(u)-\zeta p(E)|^2 )^{\frac{1}{2}}\\
&\leq ( (s+\zeta)^{2} + 4r^2 )^{\frac{1}{2}}\\
&= ( (s+\zeta)^{2} + 4\Delta^2 s^2 )^{\frac{1}{2}}\\
&= (s+\zeta)(1+4\Delta^{2}s^2/(s+\zeta)^{2} )^{\frac{1}{2}}\\
&\leq (s+\zeta)(1+16\Delta^2)^{\frac{1}{2}}\\
&\leq (s+\zeta)(1+8\Delta^2)\\
&\leq s+\zeta+16s\Delta^{2}.
\end{align*}
Putting together these estimates gives:
\begin{align*}
\mathrm{Lip}_{\mathbb{H}}(\varphi)&\leq \mathrm{Lip}_{\mathbb{H}}(\gamma)/(s+\zeta)\\
&\leq (s+\zeta+816s\Delta^2)/(s+\zeta)\\
& \leq 1+1632\Delta^2.
\end{align*}
Hence $\mathrm{Lip}_{\mathbb{H}}(\varphi)\leq 1+\eta \Delta$ since $0< \Delta< \eta/1632$. This proves (A).

\medskip

\emph{Proof of (B).} The decomposition \eqref{orthogonaldecomposition} implies:
\[|p((sE(0))\delta_{r}(u)) - (s+\zeta)p(E)| \leq 2r = 2\Delta s.\]
Since $E \in V$ and $(a,b,c)=(sE(0))\delta_{r}(u)$ we deduce:
\[|(a,b,0)-(s+\zeta)E(0)| \leq 2\Delta s.\]
Using (B') shows that for $t\in [0,1]\setminus \{1/2\}$:
\begin{align*}
|\gamma'(t) - (s+\zeta)E(0)|&\leq 2\Delta s + \frac{c}{L}(1+4L^{2})^{\frac{1}{2}}\\
&\leq 2\Delta s + 10\Delta s (1+16s^2)^{\frac{1}{2}}\\
&\leq 92 \Delta s.
\end{align*}
Hence for $t\in [0,s+\zeta]\setminus \{ (s+\zeta)/2  \}$:
\begin{align*}
|(p\circ \varphi)'(t) - p(E)|&\leq |\varphi'(t) - E(0)|\\
&\leq 92\Delta s/(s+\zeta)\\
&\leq 184\Delta.
\end{align*}
This verifies (B). The final statement of the lemma is clear from the construction.
\end{proof}

\section{Almost maximality of directional derivatives implies Pansu differentiability}\label{sectionalmostmaxpansu}
In this section we first estimate how Lipschitz horizontal curves with the same starting point and moving in similar directions stay close together (Lemma \ref{closedirectioncloseposition}). We then give simple properties of the maps $x\mapsto \langle x, E(0) \rangle$ (Lemma \ref{lemmascalarlip}) and quote a mean value estimate (Lemma \ref{preissmeanvalue}) by Preiss \cite[Lemma 3.4]{Pre90}. Finally we show that existence of an almost locally maximal horizontal directional derivative implies Pansu differentiability (Theorem \ref{almostmaximalityimpliesdifferentiability}).

\begin{lemma}\label{closedirectioncloseposition}
Given $S>0$, there is a constant $C_{\mathrm{a}}=C_{\mathrm{angle}}(S)\geq 1$ for which the following is true. Suppose:
\begin{itemize}
\item $g, h \colon I \to \mathbb{H}^{n}$ are Lipschitz horizontal curves with $\mathrm{Lip}_{\mathbb{H}}(g), \mathrm{Lip}_{\mathbb{H}}(h)\leq S$,
\item $g(c)=h(c)$ for some $c\in I$,
\item there exists $0\leq A\leq 1$ such that $|(p\circ g)'(t) - (p\circ h)'(t)| \leq A$ for almost every $t\in I$.
\end{itemize}
Then $d(g(t), h(t))\leq C_{\mathrm{a}}\sqrt{A}|t-c|$ for every $t\in I$.
\end{lemma}

\begin{proof}
Assume $c=0\in I$ and, using left group translations, $g(0)=h(0)=0$. We estimate using the equivalent Koranyi distance \eqref{Koranyi}:
\begin{align}\label{angleestimate}
d(g(t),h(t))&\leq C d_{K}(g(t),h(t)) \\
&\leq C \sum_{i=1}^{2n} |h_{i}(t)-g_{i}(t)|\nonumber \\
&\quad +C \Big| h_{2n+1}(t) - g_{2n+1}(t) + 2\sum_{i=1}^{n} (g_{i}(t)h_{n+i}(t)-g_{n+i}(t)h_{i}(t)) \Big|^{\frac{1}{2}}.\nonumber
\end{align}
Let $1\leq j\leq 2n$. Using $|(p\circ g)'(t) - (p\circ h)'(t)| \leq A$ almost everywhere implies $|h_{j}(t)-g_{j}(t)|\leq A|t|$ for every $t\in I$. Lemma \ref{lipschitzhorizontal} and $\mathrm{Lip}_{\mathbb{H}}(g)\leq S$ give the inequality $\mathrm{Lip}_{\mathbb{E}}(g_{j})\leq S$. Using also $g(0)=0$ then gives $|g_{j}(t)|\leq S|t|$ for $t\in I$. For $1\leq i\leq n$ and $t\in I$:
\begin{align*}
|g_{i}(t)h_{n+i}(t)-g_{n+i}(t)h_{i}(t)| &= |g_{i}(t)(h_{n+i}(t)-g_{n+i}(t)) + g_{n+i}(t)(g_{i}(t)-h_{i}(t))| \\
&\leq S|t| |h_{n+i}(t)-g_{n+i}(t)| + S|t| |g_{i}(t)-h_{i}(t)|\\
&\leq SAt^2.
\end{align*} 
We estimate the final term using Lemma \ref{lift}:
\begin{align*}
|h_{2n+1}(t)-g_{2n+1}(t)| & \leq 2\sum_{i=1}^{n} \int_{0}^{t} |(h_{i}'-g_{i}')(h_{n+i}-g_{n+i}) - (h_{n+i}'-g_{n+i}')(h_{i}-g_{i})| \\
& \leq 4A^2 \sum_{i=1}^{n} \int_{0}^{t} s \dd s\\
& = 2nA^2t^2.
\end{align*}
Combining our estimates of each term in \eqref{angleestimate} gives $d(g(t),h(t))\leq C\sqrt{A}|t|$ for $t\in I$, where $C$ is a constant depending on $S$.
\end{proof}

We will use the maps $x\mapsto \langle x, E(0) \rangle$ for $E\in V$ both as Pansu derivatives and as perturbations to construct an almost maximal directional derivative in the proof of Proposition \ref{DoreMaleva}. We now give simple properties of these maps.

\begin{lemma}\label{lemmascalarlip}
Suppose $E\in V$ with $\omega(E)=1$ and let $L\colon \mathbb{H}^n \to \mathbb{R}$ be the function $L(x)=\langle x, E(0) \rangle$. Then:
\begin{enumerate}
\item $L$ is $\mathbb{H}$-linear and $\mathrm{Lip}_{\mathbb{H}}(L) = 1$,
\item for $x\in\mathbb{H}^n$ and $\widetilde E\in V$:
\[\widetilde{E}L(x)=L(\widetilde{E}(0))=\langle p(\widetilde{E}), p(E) \rangle.\]
\end{enumerate}
\end{lemma}

\begin{proof}
Since $E\in V$ we know that the final coordinate of $E(0)$ is $0$. Suppose $x, y\in \mathbb{H}^n$ and $r>0$. In the first $2n$ coordinates the group product and dilations are Euclidean, hence:
\[L(x)+L(y)=\langle x+y, E(0) \rangle =  \langle xy, E(0) \rangle =L(xy)\]
and
\[L(\delta_{r}(x))=\langle \delta_{r}(x), E(0) \rangle = r\langle x, E(0) \rangle = rL(x).\]
This shows that $L$ is $\mathbb{H}$-linear.

It follows from $\mathbb{H}$-linearity that $L(x)-L(y)= L(xy^{-1})$. To bound the Lipschitz constant from above it is enough to use $\omega(E)=1$ to observe:
\[|\langle x, E(0) \rangle |\leq |p(x)|\leq d(x).\]
Conversely, $L(E(0))-L(0)=\langle E(0), E(0) \rangle =1$, so equals $d(E(0))$ by Lemma \ref{horizontaldistances}. This proves (1).

To prove (2) we observe:
\begin{align*}
\widetilde{E}L(x)&=\lim_{t\to 0} \frac{L(x + t\widetilde{E}(x)) - L(x)}{t}\\
&=\lim_{t\to 0}\Big\langle \frac{x+t \widetilde{E}(x) -x}{t},E(0) \Big\rangle\\
&=\langle \widetilde{E}(x),E(0)\rangle\\
&=\langle \widetilde{E}(0), E(0)\rangle.
\end{align*}
\end{proof}

A key feature of the special pairs used to define almost maximal directional derivatives will be that changes in difference quotients are bounded by changes in directional derivatives. We use the following lemma \cite[Lemma 3.4]{Pre90}.

\begin{lemma}\label{preissmeanvalue}
Suppose $|\zeta|<s<\rho$, $0<v<1/32$, $\sigma>0$ and $L>0$ are real numbers. Let $\varphi, \psi\colon \mathbb{R} \to \mathbb{R}$ satisfy $\mathrm{Lip}_{\mathbb{E}}(\varphi)+\mathrm{Lip}_{\mathbb{E}}(\psi)\leq L$, $\varphi(t)=\psi(t)$ for $|t|\geq s$ and $\varphi(\zeta)\neq \psi(\zeta)$. Suppose, moreover, that $\psi'(0)$ exists and that
\[|\psi(t)-\psi(0)-t\psi'(0)|\leq \sigma L|t|\]
whenever $|t|\leq \rho$,
\[\rho\geq s\sqrt{(sL)/(v|\varphi(\zeta)-\psi(\zeta)|)},\]
and
\[\sigma \leq v^{3}\Big( \frac{\varphi(\zeta)-\psi(\zeta)}{sL} \Big)^{2}.\]
Then there is $\tau\in (-s,s)\setminus \{\zeta\}$ such that $\varphi'(\tau)$ exists,
\[\varphi'(\tau)\geq \psi'(0)+v|\varphi(\zeta)-\psi(\zeta)|/s,\]
and
\[|(\varphi(\tau+t)-\varphi(\tau))-(\psi(t)-\psi(0))|\leq 4(1+20v)\sqrt{(\varphi'(\tau)-\psi'(0))L}|t|\]
for every $t\in \mathbb{R}$.
\end{lemma}

\begin{remark}\label{meanvalueremark}
By examining the proof of Lemma \ref{preissmeanvalue} in \cite{Pre90} it is easy to see that $\tau$ can additionally be chosen outside a given Lebesgue measure zero subset of $\mathbb{R}$. A stronger observation, that $\tau$ can be chosen outside a set of sufficiently small yet positive measure, is used in \cite{Pre90} to prove \cite[Theorem 6.3]{Pre90}.
\end{remark}

We can now prove that existence of an almost locally maximal horizontal directional derivative implies Pansu differentiability. The argument is based on that of \cite[Theorem 4.1]{Pre90}, but we use our analysis of the Carnot-Carath\'{e}odory distance and use exclusively horizontal curves and directional derivatives in horizontal directions. 

\begin{notation}\label{D^f}
Fix a Lebesgue null $G_{\delta}$ set $N\subset \mathbb{H}^n$ as in Lemma \ref{uds} for the remainder of the article. For any Lipschitz function $f:\mathbb{H}^n\to \mathbb{R}$ define:
\[D^{f}:=\{ (x,E) \in N\times V\colon \omega(E)=1,\, Ef(x) \mbox{ exists}\}.\]
\end{notation}

\begin{theorem}\label{almostmaximalityimpliesdifferentiability}
Let $f\colon \mathbb{H}^n\to \mathbb{R}$ be a Lipschitz function with $\mathrm{Lip}_{\mathbb{H}}(f) \leq 1/2$. Suppose $(x_{\ast}, E_{\ast})\in D^{f}$. Let $M$ denote the set of pairs $(x,E)\in D^{f}$ such that $Ef(x)\geq E_{\ast}f(x_{\ast})$ and
\begin{align*}
& |(f(x + tE_{\ast}(x))-f(x)) - (f(x_{\ast} + tE_{\ast}(x_{\ast}))-f(x_{\ast}))| \\
& \qquad \leq 6|t| (  (Ef(x)-E_{\ast}f(x_{\ast}))\mathrm{Lip}_{\mathbb{H}}(f))^{\frac{1}{4}}
\end{align*}
for every $t\in (-1,1)$. If
\[\lim_{\delta \downarrow 0} \sup \{Ef(x)\colon (x,E)\in M \mbox{ and }d(x,x_{\ast})\leq \delta\}\leq E_{\ast}f(x_{\ast})\]
then $f$ is Pansu differentiable at $x_{\ast}$ with Pansu derivative
\[L(x)=E_{\ast}f(x_{\ast})\langle x , E_{\ast}(0) \rangle=E_{\ast}f(x_{\ast})\langle p(x) , p(E_{\ast}) \rangle.\]
\end{theorem}

\begin{remark}
Since we will apply Lemma \ref{preissmeanvalue}, it may seem more intuitive to instead bound $|(f(x + tE(x))-f(x)) - (f(x_{\ast}+tE_{\ast}(x_{\ast}))-f(x_{\ast}))|$ in the statement of Theorem \ref{almostmaximalityimpliesdifferentiability}. The precise form in Theorem \ref{almostmaximalityimpliesdifferentiability} will be useful when we construct an almost locally maximal horizontal directional derivative in Proposition \ref{DoreMaleva}.
\end{remark}

We will prove Theorem \ref{almostmaximalityimpliesdifferentiability} by contradiction. We first use Lemma \ref{newcurveg} to modify the line $x_{\ast} + tE_{\ast}(x_{\ast})$ to form a Lipschitz horizontal curve $g$ in $N$ which passes through a nearby point showing non-Pansu differentiability at $x_{\ast}$. We then apply Lemma \ref{preissmeanvalue} with $\varphi=f\circ g$ to obtain a large directional derivative along $g$ and estimates for difference quotients in the new direction. We then develop these estimates to show that the new point and direction form a pair in $M$. This shows that there is a nearby pair in $M$ giving a larger directional derivative than $(x_{\ast},E_{\ast})$, a contradiction.

\begin{proof}[Proof of Theorem \ref{almostmaximalityimpliesdifferentiability}]
We can assume $\mathrm{Lip}_{\mathbb{H}}(f)>0$ since otherwise the statement is trivial. Let $\varepsilon>0$ and fix various parameters as follows.

\vspace{0.2cm}

\emph{Parameters.} Choose:
\begin{enumerate}
\item $0< v<1/32$ such that $4(1+20v)\sqrt{(2+v)/(1-v)}+v < 6$,
\item $\eta=\varepsilon v^{3}/3200$,
\item $\Delta(\eta/2)$, $C_{\mathrm{m}}$ and $C_{\mathrm{a}}=C_{\mathrm{angle}}(2)$ using Lemma \ref{newcurveg} and Lemma \ref{closedirectioncloseposition},
\item rational $0< \Delta < \min \{\eta v^2,\, \Delta(\eta/2),\, \varepsilon v^{5}/(8C_{\mathrm{m}}^{2}C_{\mathrm{a}}^{4}\mathrm{Lip}_{\mathbb{H}}(f)^3) \}$,
\item $\sigma=9\varepsilon^{2}v^{5}\Delta^2/256$,
\item $0<\rho<1$ such that
\begin{equation}\label{directionaldifferentiability}
|f(x_{\ast} + tE_{\ast}(x_{\ast})) - f(x_{\ast})-tE_{\ast}f(x_{\ast})|\leq \sigma \mathrm{Lip}_{\mathbb{H}}(f)|t|
\end{equation}
for every $|t|\leq \rho$,
\item $0<\delta < \rho \sqrt{3\varepsilon v\Delta^{3}}/4$ such that
\[Ef(x)<E_{\ast}f(x_{\ast})+\varepsilon v\Delta/2\]
whenever $(x,E)\in M$ and $d(x,x_{\ast})\leq 4\delta(1+1/\Delta)$.
\end{enumerate}

To prove Pansu differentiability of $f$ at $x_{\ast}$ we will show:
\[|f(x_{\ast}\delta_{t}(h))-f(x_{\ast})-tE_{\ast}f(x_{\ast})\langle h, E_{\ast}(0) \rangle |\leq \varepsilon t\]
whenever $d(h) \leq 1$ and $0<t<\delta$. Suppose this is not true. Then there exists $u\in \mathbb{Q}^{2n+1}$ with $d(u) \leq 1$ and rational $0<r<\delta$ such that:
\begin{equation}\label{badpoint}
|f(x_{\ast}\delta_{r}(u))-f(x_{\ast})-rE_{\ast}f(x_{\ast})\langle u, E_{\ast}(0) \rangle|> \varepsilon r.
\end{equation}
Let $s=r/ \Delta \in \mathbb{Q}$. We next construct Lipschitz horizontal curves $g$ and $h$ for which we can apply Lemma \ref{preissmeanvalue} with $\varphi:=f\circ g$ and $\psi:=f\circ h$.

\medskip

\emph{Construction of $g$.} To ensure that the image of $g$ is a subset of the set $N$, we first introduce rational approximations to $x_{\ast}$ and $E_{\ast}$. 

Since the Carnot-Carath\'{e}odory and Euclidean distances are topologically equivalent, $\mathbb{Q}^{2n+1}$ is dense in $\mathbb{R}^{2n+1}$ with respect to the distance $d$. The set
\[ \{E\in V\colon \omega(E)=1, \, E \mbox{ a rational linear combination of }X_{i}, Y_{i},\, 1\leq i\leq n\}\]
is dense in $\{E\in V\colon \omega(E)=1\}$ with respect to the norm $\omega$. To see this, suppose $E\in V$ satisfies $\omega(E)=1$. The Euclidean sphere $\mathbb{S}^{2n-1} \subset \mathbb{R}^{2n}$ contains a dense set $S$ of points with rational coordinates. This fact is well known, e.g. one can use stereographic projection. Let $q=(q_1,\ldots, q_{2n})\in \mathbb{S}^{2n-1}$ be the coefficients of $E$ in the basis $\{X_{i}, Y_{i}: 1\leq i\leq n\}$. Take $\widetilde{q}\in S$ such that $| q-\widetilde{q}|$ is small and define the rational approximation of $E$ as the linear combination of $\{X_{i}, Y_{i}: 1\leq i\leq n\}$ with coefficients $\widetilde{q}_i$. 

Define
\begin{equation}\label{A1}A_{1}=(\eta \Delta/C_{\mathrm{a}})^{2}\end{equation}
and
\begin{equation}\label{A2}A_{2}=\Big(6- \Big( 4(1+20v) \Big( \frac{2+v}{1-v} \Big)^{\frac{1}{2}}+v\Big)\Big)^{2} \frac{(  \varepsilon v\Delta/2   )\mathrm{Lip}_{\mathbb{H}}(f))^{\frac{1}{2}}}{C_{\mathrm{a}}^{2}\mathrm{Lip}_{\mathbb{H}}(f)^{2}}.\end{equation}
Notice $A_{1}, A_{2}>0$ using, in particular, our choice of $v$. Choose $\widetilde{x}_{\ast}\in \mathbb{Q}^{2n+1}$ and $\widetilde{E}_{\ast}\in V$ with $\omega( \widetilde{E}_{\ast})=1$, a rational linear combination of $\{X_{i}, Y_{i}: 1\leq i\leq n\}$, sufficiently close to $x_{\ast}$ and $E_{\ast}$ to ensure:
\begin{equation}\label{nowlistingthese} d(\widetilde{x}_{\ast}\delta_{r}(u),x_{\ast})\leq 2r,\end{equation}
\begin{equation}\label{lista} d(\widetilde{x}_{\ast}\delta_{r}(u), x_{\ast}\delta_{r}(u))\leq \sigma r,\end{equation}
\begin{equation}\label{listvector1} \omega(\widetilde{E}_{\ast}-E_{\ast})\leq \min \{ \sigma, \, C_{\mathrm{m}}\Delta,\, A_{1},\, A_{2} \},\end{equation}
which is possible since all terms on the right side of the above inequalities are strictly positive.
Additionally, we choose $\widetilde{x}_{\ast}$ and $\widetilde{E}_{\ast}$ sufficiently close to $x_{\ast}$ and $E_{\ast}$ so that if $a,b \in \mathbb{R}^{n}$ and $c\in \mathbb{R}$ are defined by
\[(a,b,c)=(x_{\ast}+ (s/2)E_{\ast}(x_{\ast}))^{-1}(\widetilde{x}_{\ast}+s\widetilde{E}_{\ast}(\widetilde{x}_{\ast})),\]
then:
\begin{equation}\label{weird1}(a^2+b^2)^{\frac{1}{2}}\left(1+ \frac{c^2}{(a^2+b^2)^2} + \frac{4c^2}{(a^2+b^2)} \right)^{\frac{1}{2}}\leq \frac{s}{2}\left(1+\frac{\eta \Delta}{2}\right)\end{equation}
and
\begin{equation}\label{weird2} \frac{c}{(a^2+b^2)^{\frac{1}{2}}}(1+4(a^2+b^2))^{\frac{1}{2}}\leq \frac{s}{2} \min \{A_{1},\, A_{2}\}.\end{equation}
To show these requirements can be satisfied, we remark that
\begin{align*}
(x_{\ast}+ (s/2)E_{\ast}(x_{\ast}))^{-1}(x_{\ast}+sE_{\ast}(x_{\ast})) &= (x_{\ast}\exp((s/2)E_{\ast}))^{-1} (x_{\ast}\exp(sE_{\ast}))\\
&= \exp(-(s/2)E_{\ast})x_{\ast}^{-1}x_{\ast}\exp(sE_{\ast})\\
&=\exp((s/2)E_{\ast})\\
&=(s/2)E_{\ast}(0),
\end{align*}
which has final coordinate zero and satisfies $|p((s/2)E_{\ast}(0))|=s/2$.
We also choose $\widetilde{x}_{\ast}$ and $\widetilde{E}_{\ast}$ so that \eqref{weird1} and \eqref{weird2} hold if instead:
\[(a,b,c)=(\widetilde{x}_{\ast} -s\widetilde{E}_{\ast}(\widetilde{x}_{\ast}))^{-1} (x_{\ast} -(s/2)E_{\ast}(x_{\ast})),\]
with a similar justification.

Note $0<r<\Delta$ and recall $s=r/\Delta$ is rational. To construct $g$ we apply Lemma \ref{newcurveg} with the following parameters:
\begin{itemize}
\item $\eta, r, \Delta$ and $u$ as defined above in \eqref{badpoint},
\item $x=\widetilde{x}_{\ast}$ and $E=\widetilde{E}_{\ast}$.
\end{itemize}
This gives a Lipschitz horizontal curve $g\colon \mathbb{R} \to \mathbb{H}^{n}$ with the following properties:
\begin{enumerate}
\item $g(t)=\widetilde{x}_{\ast} + t\widetilde{E}_{\ast}(x_{\ast})$ for $|t|\geq s$,
\item $g(\zeta)=\widetilde{x}_{\ast}\delta_{r}(u)$, where $\zeta := r\langle u,\widetilde{E}_{\ast}(0)\rangle$,
\item $\mathrm{Lip}_{\mathbb{H}}(g)\leq 1+\eta\Delta/2$,
\item $g'(t)$ exists and $|(p\circ g)'(t) - p(\widetilde{E}_{\ast})| \leq C_{\mathrm{m}}\Delta$ for $t\in \mathbb{R}$ outside a finite set.
\end{enumerate}
Since the relevant quantities were rational and the set $N$ was chosen using Lemma \ref{uds}, we also know that the image of $g$ is contained in $N$.

\medskip

\emph{Construction of $h$.} Denote $q_1=x_{\ast}+ (s/2)E_{\ast}(x_{\ast})$ and $q_2=\widetilde{x}_{\ast}+s\widetilde{E}_{\ast}(\widetilde{x}_{\ast})$. Applying Lemma \ref{goodcurve} with $(a,b,c)=q_{1}^{-1}q_{2}$ and using inequalities \eqref{weird1} and \eqref{weird2} gives $\widetilde{h}_{1}\colon [0,1]\to \mathbb{H}^{n}$ such that:
\begin{enumerate}
\item $\widetilde{h}_{1}$ is a Lipschitz horizontal curve joining $0$ to $q_{1}^{-1}q_{2}$,
\item $\mathrm{Lip}_{\mathbb{H}}(\widetilde{h}_{1})\leq (s/2)(1+\eta\Delta/2)$,
\item $\widetilde{h}_{1}'(t)$ exists and satisfies
\[|(p\circ \widetilde{h}_{1})'(t)-(s/2)p(E_{\ast})|\leq (s/2) \min \{A_{1}, A_{2}\}\]
for $t\in [0,1]\setminus \{1/2\}$.
\end{enumerate}
Define $h_{1}\colon [s/2,s]\to \mathbb{H}^{n}$ by:
\[h_{1}(t)=q_{1}\widetilde{h}_{1}((2/s)(t-s/2)).\]
Then $h_{1}$ is a Lipschitz horizontal curve joining $q_1$ to $q_2$ with
\[\mathrm{Lip}_{\mathbb{H}}(h_{1})\leq 1+\eta\Delta/2.\]
The derivative $h_{1}'(t)$ exists and satisfies
\[|(p\circ h_{1})'(t)- p(E_{\ast})|\leq \min \{A_{1}, A_{2}\}\]
for $t\in [s/2,s]\setminus \{3s/4\}$, since left translations act linearly in the first $2n$ coordinates.

Similarly, there is a Lipschitz horizontal curve $h_{2}\colon [-s,-s/2]\to \mathbb{H}^{n}$ joining $\widetilde{x}_{\ast} -s\widetilde{E}_{\ast}(\widetilde{x}_{\ast})$ to $x_{\ast} -(s/2)E_{\ast}(x_{\ast})$ satisfying
\[\mathrm{Lip}_{\mathbb{H}}(h_{2})\leq 1+\eta\Delta/2\]
and
\[|(p\circ h_{2})'(t)- p(E_{\ast})|\leq \min \{A_{1}, A_{2}\}\]
for $t\in [-s,-s/2]\setminus \{-3s/4\}$. 

Define a Lipschitz horizontal curve $h\colon \mathbb{R}\to \mathbb{H}^{n}$ by:
\[h(t)= \begin{cases} \widetilde{x}_{\ast} + t\widetilde{E}_{\ast}(\widetilde{x}_{\ast}) &\mbox{if }|t|\geq s,\\
x_{\ast} + tE_{\ast}(x_{\ast}) &\mbox{if }|t|\leq s/2,\\
h_{1}(t) &\mbox{if }s/2<t<s,\\
h_{2}(t) &\mbox{if }-s<t<-s/2.\end{cases}\]
Using the inequalities $\mathrm{Lip}_{\mathbb{H}}(h_{1}), \mathrm{Lip}_{\mathbb{H}}(h_{2})\leq 1+\eta\Delta/2$ and $\omega(E_{\ast}), \omega(\widetilde{E}_{\ast})\leq 1$ gives:
\[\mathrm{Lip}_{\mathbb{H}}(h)\leq 1+\eta\Delta/2.\]
Also $h'(t)$ exists for $t\in \mathbb{R}\setminus \{\pm  3s/4, \pm s/2, \pm s\}$ and
\[|(p\circ h)'(t)-p(E_{\ast})|\leq \min \{A_{1}, A_{2} \},\]
using the corresponding bounds for $h_{1}, h_{2}$ and $\omega(E_{\ast}-\widetilde{E}_{\ast})\leq \min \{A_{1}, A_{2} \}$ from \eqref{listvector1}.

\medskip

\emph{Application of Lemma \ref{preissmeanvalue}.} We now prove that the assumptions of Lemma \ref{preissmeanvalue} hold with $L:=(2+\eta \Delta)\mathrm{Lip}_{\mathbb{H}}(f)$, $\varphi:=f\circ g$ and $\psi:=f\circ h$. The inequalities $|\zeta|<s<\rho$, $0<v<1/32$ and the equality $\varphi(t)=\psi(t)$ for $|t|\geq s$ are clear. Using $\mathrm{Lip}_{\mathbb{H}}(g), \mathrm{Lip}_{\mathbb{H}}(h) \leq 1+\eta\Delta/2$ gives $\mathrm{Lip}_{\mathbb{E}}(\varphi)+\mathrm{Lip}_{\mathbb{E}}(\psi)\leq L$. 

Notice that \eqref{lista} implies:
\[ |f(\widetilde{x}_{\ast}\delta_{r}(u)) - f(x_{\ast}\delta_{r}(u))| \leq \sigma r \mathrm{Lip}_{\mathbb{H}}(f).\]
Since $|\zeta|\leq r\leq \rho$, we may substitute $t=\zeta$ in \eqref{directionaldifferentiability} to obtain
\begin{align*}
|f(x_{\ast} + \zeta E_{\ast}(x_{\ast}))-f(x_{\ast})-\zeta E_{\ast}f(x_{\ast})| &\leq \sigma \mathrm{Lip}_{\mathbb{H}}(f)|\zeta| \\
&\leq \sigma r\mathrm{Lip}_{\mathbb{H}}(f).
\end{align*}
Next note that \eqref{listvector1} implies $|\widetilde{E}_{\ast}(0)-E_{\ast}(0)|\leq \sigma$. We use also $\zeta=r\langle u, \widetilde{E}_{\ast}(0)\rangle$ to estimate as follows:
\begin{align*}
|\zeta E_{\ast}f(x_{\ast}) - r\langle u,E_{\ast}(0)\rangle E_{\ast}f(x_{\ast})| & = r|E_{\ast}f(x_{\ast})||\langle u, \widetilde{E}_{\ast}(0) -E_{\ast}(0)\rangle|\\
&\leq r\mathrm{Lip}_{\mathbb{H}}(f)|\widetilde{E}_{\ast}(0)-E_{\ast}(0)|\\
&\leq \sigma r \mathrm{Lip}_{\mathbb{H}}(f).
\end{align*}
Hence we obtain,
\begin{equation}\label{yetanother}|f(x_{\ast} + \zeta E_{\ast}(x_{\ast}))-f(x_{\ast})-r\langle u,E_{\ast}(0)\rangle E_{\ast}f(x_{\ast})| \leq 2\sigma r\mathrm{Lip}_{\mathbb{H}}(f).\end{equation}

Notice $|\zeta|\leq r=\Delta s\leq s/2$, so the definition of the $h$ gives $h(\zeta)=x_{\ast}+\zeta E_{\ast}(x_{\ast})$. The definition of $g$ gives $g(\zeta)=\widetilde{x}_{\ast}\delta_{r}(u)$. Using also \eqref{badpoint} and \eqref{yetanother}, we can estimate as follows:
\begin{align}\label{differenceofcomposition}
|\varphi(\zeta)-\psi(\zeta)|&= |f(g(\zeta)) - f(h(\zeta))| \nonumber \\
&= |f(\widetilde{x}_{\ast}\delta_{r}(u)) - f(x_{\ast}+\zeta E_{\ast}(x_{\ast}))| \nonumber\\
& \geq |f(x_{\ast}\delta_{r}(u)) - f(x_{\ast} + \zeta E_{\ast}(x_{\ast}))| - |f(\widetilde{x}_{\ast}\delta_{r}(u)) - f(x_{\ast}\delta_{r}(u))|\nonumber \\
& \geq |f(x_{\ast}\delta_{r}(u))-f(x_{\ast})-rE_{\ast}f(x_{\ast})\langle u, E_{\ast}(0) \rangle | \nonumber \\
& \quad -|f(x_{\ast} + \zeta E_{\ast}(x_{\ast})) - f(x_{\ast}) - r E_{\ast}f(x_{\ast})\langle u, E_{\ast}(0) \rangle| \nonumber \\
& \quad - \sigma r\mathrm{Lip}_{\mathbb{H}}(f) \nonumber\\
&\geq \varepsilon r - 2\sigma r\mathrm{Lip}_{\mathbb{H}}(f) -  \sigma r\mathrm{Lip}_{\mathbb{H}}(f)\nonumber \\
&= \varepsilon r - 3\sigma r\mathrm{Lip}_{\mathbb{H}}(f) \nonumber \\
&\geq 3\varepsilon r/4.
\end{align}
In particular, $\varphi(\zeta)\neq \psi(\zeta)$.

We next check that $\psi'(0)$ exists and
\begin{equation}\label{psiprime}
|\psi(t)-\psi(0)-t\psi'(0)| \leq \sigma L|t|
\end{equation}
whenever $|t|\leq \rho$. Notice $\psi'(0)$ exists and equals $E_{\ast}f(x_{\ast})$, since $E_{\ast}f(x_{\ast})$ exists and $\psi(t)=f(x_{\ast}+tE_{\ast}(x_{\ast}))$ for $|t|\leq s/2$. Recall $|(p\circ h)'-p(E_{\ast})|\leq A_{1}$ and also the definition $A_{1}=(\eta\Delta/C_{\mathrm{a}})^{2}$ from \eqref{A1}. Since $h(0)=x_{\ast}$, Lemma \ref{closedirectioncloseposition} implies that
\[ d(x_{\ast}+tE_{\ast}(x_{\ast}),h(t))\leq C_{\mathrm{a}}\sqrt{A_{1}}|t|\leq \eta\Delta |t|.\]
Hence, using also \eqref{directionaldifferentiability} and $L=(2+\eta\Delta)\mathrm{Lip}_{\mathbb{H}}(f)$,
\begin{align*}
|\psi(t)-\psi(0)-t\psi'(0)| &\leq |f(x_{\ast} + tE_{\ast}(x_{\ast})) - f(x_{\ast})-tE_{\ast}f(x_{\ast})|\\
& \qquad + |f(x_{\ast} + tE_{\ast}(x_{\ast})) - f(h(t))|\\
&\leq \sigma \mathrm{Lip}_{\mathbb{H}}(f)|t| + \mathrm{Lip}_{\mathbb{H}}(f)\eta\Delta|t|\\
&\leq \sigma L |t|.
\end{align*}

Recall that $\mathrm{Lip}_{\mathbb{H}}(f)\leq 1/2$, which implies $L\leq 4$. Using also $r< \delta$, $s=r/\Delta$, \eqref{differenceofcomposition} and the definition of $\delta$ in Parameters (7) we deduce:
\begin{align*}
s\sqrt{ sL/(v|\varphi(\zeta)-\psi(\zeta)|)}  &\leq 4s\sqrt{s/(3\varepsilon rv)}\\
&= 4r/\sqrt{3\varepsilon v\Delta^3}\\
&\leq 4\delta/ \sqrt{3\varepsilon v\Delta^{3}}\\
&\leq \rho.
\end{align*}

Finally we use \eqref{differenceofcomposition}, $L\leq 4$ and the definition of $\sigma$ in Parameters (5) to observe:
\begin{align*}
v^3 (|\varphi(\zeta)-\psi(\zeta)|/(sL))^2&\geq v^3(3\varepsilon r / 16s)^2\\
&= 9\varepsilon^2 v^3 \Delta^2 /256\\
& \geq \sigma.
\end{align*}

We may now apply Lemma \ref{preissmeanvalue}. We obtain $\tau \in (-s,s)\setminus \{\zeta \}$ such that $\varphi'(\tau)$ exists,
\begin{equation}\label{bigderivative}
\varphi'(\tau)\geq \psi'(0)+v|\varphi(\zeta)-\psi(\zeta)|/s,
\end{equation}
and
\begin{equation}\label{incrementsbound}
|(\varphi(\tau+t)-\varphi(\tau))-(\psi(t)-\psi(0))|\leq 4(1+20v)\sqrt{(\varphi'(\tau)-\psi'(0))L}|t|
\end{equation}
for every $t\in \mathbb{R}$. Since $g$ is a horizontal curve, we may use Remark \ref{meanvalueremark} to additionally choose $\tau$ such that $g'(\tau)$ exists and is in $\mathrm{Span}\{X_{i}(g(\tau)), Y_{i}(g(\tau))\colon 1\leq i\leq n\}$. 

\medskip

\emph{Conclusion.} Let $x=g(\tau)\in N$ and choose $E\in V$ with $E(g(\tau))=g'(\tau)/|p(g'(\tau))|$, which implies that $\omega(E)=1$. We will transform \eqref{bigderivative} and \eqref{incrementsbound} into
\begin{equation}\label{betterpair1}
Ef(x)\geq E_{\ast}f(x_{\ast}) + \varepsilon v\Delta/2
\end{equation}
and
\begin{equation}\label{betterpair2}
(x,E)\in M.
\end{equation}
We first observe that this suffices to conclude the proof. Recall $d(\widetilde{x}_{\ast}\delta_{r}(u),x_{\ast})\leq 2r$ from \eqref{nowlistingthese}. Using $g(\tau)=x$ and $g(\zeta)=\widetilde{x}_{\ast}\delta_{r}(u)$ gives:
\begin{align*}
d(x,x_{\ast}) &\leq d(g(\tau),g(\zeta))+d(\widetilde{x}_{\ast}\delta_{r}(u),x_{\ast})\\
&\leq \mathrm{Lip}_{\mathbb{H}}(g)|\tau - \zeta| +2r\\
&\leq 4(s+r)\\
&= 4r(1+1/\Delta)\\
&\leq 4\delta(1+1/\Delta).
\end{align*}
Since $x\in N$, combining this with \eqref{betterpair1} and \eqref{betterpair2} contradicts the choice of $\delta$ in Parameters (7). This contradiction forces us to conclude that \eqref{badpoint} is false, finishing the proof.

\medskip

\emph{Proof of \eqref{betterpair1}.} Using \eqref{differenceofcomposition} and \eqref{bigderivative} we see:
\begin{equation}\label{stanco}
\varphi'(\tau)-\psi'(0)\geq 3\varepsilon vr/4s=3\varepsilon v\Delta/4.
\end{equation}
Notice that $\varphi'(\tau)=Ef(x)|p(g'(\tau))|$ using Definition \ref{defdirectionalderivative} and our choice of $E$. Since $\omega(E)=1$ implies $|Ef(x)|\leq \mathrm{Lip}_{\mathbb{H}}(f)$, we deduce $|\varphi'(\tau)|/|p(g'(\tau))|\leq \mathrm{Lip}_{\mathbb{H}}(f)$. Also $|p(g'(\tau))| \leq \mathrm{Lip}_{\mathbb{H}}(g)\leq 1+\eta \Delta$. Using $\psi'(0)=E_{\ast}f(x_{\ast})$ and \eqref{stanco} gives:
\begin{align*}
& Ef(x)-E_{\ast}f(x_{\ast})-(1-v)(\varphi'(\tau)-\psi'(0))\\
&\qquad = v(\varphi'(\tau)-\psi'(0)) + (1-|p(g'(\tau))|)\varphi'(\tau)/|p(g'(\tau))|\\
&\qquad \geq 3\varepsilon v^2\Delta/4 - \eta\Delta \mathrm{Lip}_{\mathbb{H}}(f)\\
&\qquad \geq 0,
\end{align*}
where in the last inequality we used $\mathrm{Lip}_{\mathbb{H}}(f)\leq 1/2$ and $\eta\leq 3\varepsilon v^2 /2$ from Parameters (2). From this we use $0<v<1/32$ and \eqref{stanco} again to deduce:
\begin{equation}\label{noimagination}
Ef(x)-E_{\ast}f(x_{\ast})\geq (1-v)(\varphi'(\tau)-\psi'(0))\geq \varepsilon v\Delta /2
\end{equation}
which proves \eqref{betterpair1}.

\medskip

\emph{Proof of \eqref{betterpair2}.} Recall that $|(p\circ g)'(t)-p(\widetilde{E}_{\ast})| \leq C_{\mathrm{m}}\Delta$ for all but finitely many $t$, from our construction of $g$. Using \eqref{listvector1}, this implies $|(p\circ g)'(t)- p(E_{\ast})|\leq 2C_{\mathrm{m}}\Delta$ for all but finitely many $t$. Since $x=g(\tau)$, we can apply Lemma \ref{closedirectioncloseposition} to obtain
\[d(g(\tau+t),x+tE_{\ast}(x))\leq C_{\mathrm{a}}\sqrt{2C_{\mathrm{m}}\Delta}|t|\]
for $t\in \mathbb{R}$. By \eqref{noimagination} we have $\Delta \leq 2(Ef(x)-E_{\ast}f(x_{\ast}))/(\varepsilon v)$. Using also the definition of $\Delta$ from Parameters (4), we deduce that for $t\in \mathbb{R}$:
\begin{align}\label{add1}
&|(f(x+tE_{\ast}(x))-f(x))-(f(g(\tau+t))-f(g(\tau)))|\nonumber \\
&\qquad = |f(x+tE_{\ast}(x))-f(g(\tau+t))|\nonumber \\
&\qquad \leq \mathrm{Lip}_{\mathbb{H}}(f)  d(g(\tau+t), x+tE_{\ast}(x))\nonumber \\
&\qquad \leq C_{\mathrm{a}}\sqrt{2C_{\mathrm{m}}\Delta} \mathrm{Lip}_{\mathbb{H}}(f)|t|\nonumber \\
&\qquad \leq C_{\mathrm{a}}\sqrt{2C_{\mathrm{m}}}\mathrm{Lip}_{\mathbb{H}}(f)|t|\Delta^{\frac{1}{4}}\Big( \frac{2(Ef(x)-E_{\ast}f(x_{\ast}))}{\varepsilon v} \Big)^{\frac{1}{4}}\nonumber \\
&\qquad \leq v|t|\big((Ef(x)-E_{\ast}f(x_{\ast}))\mathrm{Lip}_{\mathbb{H}}(f) \big)^{\frac{1}{4}} \Big(\frac{8C_{\mathrm{m}}^{2}C_{\mathrm{a}}^{4}\Delta\mathrm{Lip}_{\mathbb{H}}(f)^{3}}{\varepsilon v^{5}} \Big)^{\frac{1}{4}} \nonumber \\
&\qquad \leq v|t|\big((Ef(x)-E_{\ast}f(x_{\ast}))\mathrm{Lip}_{\mathbb{H}}(f) \big)^{\frac{1}{4}}.
\end{align}
Combining \eqref{incrementsbound}, \eqref{noimagination} and $L=(2+\eta \Delta)\mathrm{Lip}_{\mathbb{H}}(f)\leq (2+v)\mathrm{Lip}_{\mathbb{H}}(f)$ gives:
\begin{align}\label{add2}
&|(\varphi(\tau+t)-\varphi(\tau))-(\psi(t)-\psi(0))| \nonumber \\
&\qquad \leq 4(1+20v)|t| \Big( \frac{(2+v)\mathrm{Lip}_{\mathbb{H}}(f)(Ef(x)-E_{\ast}f(x_{\ast}))}{1-v} \Big)^{\frac{1}{2}}
\end{align}
for $t\in \mathbb{R}$. Using $\mathrm{Lip}_{\mathbb{H}}(f)\leq 1/2$ gives the simple bound:
\[((Ef(x)-E_{\ast}f(x_{\ast}))\mathrm{Lip}_{\mathbb{H}}(f))^{\frac{1}{2}} \leq ((Ef(x)-E_{\ast}f(x_{\ast}))\mathrm{Lip}_{\mathbb{H}}(f))^{\frac{1}{4}}\]
since both sides are less than $1$. Hence adding \eqref{add1} and \eqref{add2} and using the definition $\varphi=f\circ g$ gives:
\begin{align}\label{add3}
& |f(x+tE_{\ast}(x)-f(x))-(\psi(t)-\psi(0))|\nonumber \\
&\qquad \leq \Big( 4(1+20v) \Big( \frac{2+v}{1-v} \Big)^{\frac{1}{2}}+v\Big) |t|( (Ef(x)-E_{\ast}f(x_{\ast}))\mathrm{Lip}_{\mathbb{H}}(f))^{\frac{1}{4}}
\end{align}
for $t\in \mathbb{R}$. 

Recall $\psi = f\circ h$ and $h(0)=x_{\ast}$. Using the inequality $|(p\circ h)'-p(E_{\ast})|\leq A_{2}$, Lemma \ref{closedirectioncloseposition}, our definition of $A_{2}$ in \eqref{A2} and \eqref{betterpair1}, we can estimate as follows:
\begin{align}\label{add4}
&|(\psi(t)-\psi(0))-(f(x_{\ast}+tE_{\ast}(x_{\ast}))-f(x_{\ast}))| \nonumber \\
&\qquad =|f(h(t))-f(x_{\ast}+tE_{\ast}(x_{\ast}))| \nonumber \\
&\qquad \leq \mathrm{Lip}_{\mathbb{H}}(f)d(h(t), x_{\ast}+tE_{\ast}(x_{\ast})) \nonumber \\
&\qquad \leq \mathrm{Lip}_{\mathbb{H}}(f)C_{\mathrm{a}}\sqrt{A_{2}}|t| \nonumber \\
&\qquad =  \Big(6- \Big( 4(1+20v) \Big( \frac{2+v}{1-v} \Big)^{\frac{1}{2}}+v\Big)\Big)|t|( \varepsilon v \Delta/2)\mathrm{Lip}_{\mathbb{H}}(f))^{\frac{1}{4}}\\
&\qquad \leq  \Big(6- \Big( 4(1+20v) \Big( \frac{2+v}{1-v} \Big)^{\frac{1}{2}}+v\Big)\Big)|t|((Ef(x)-E_{\ast}f(x_{\ast}))\mathrm{Lip}_{\mathbb{H}}(f))^{\frac{1}{4}}
\end{align}
for all $t\in \mathbb{R}$. Adding \eqref{add3} and \eqref{add4} gives:
\begin{align*}
& |(f(x+tE_{\ast}(x))-f(x)) - (f(x_{\ast}+tE_{\ast}(x_{\ast}))-f(x_{\ast}))| \\
& \qquad \leq 6|t| \big(  (Ef(x)-E_{\ast}f(x_{\ast}))\mathrm{Lip}_{\mathbb{H}}(f) \big)^{\frac{1}{4}}
\end{align*}
for $t\in \mathbb{R}$. This implies \eqref{betterpair2}, hence proving the theorem.
\end{proof}

\section{Construction of an almost maximal directional derivative}\label{construction}
The main result of this section is Proposition \ref{DoreMaleva}, which is an adaptation of \cite[Theorem 3.1]{DM11} to $\mathbb{H}^n$. It shows that given a Lipschitz function $f_{0}\colon \mathbb{H}^{n} \to \mathbb{R}$, there is a Lipschitz function $f\colon \mathbb{H}^{n} \to \mathbb{R}$ such that $f-f_{0}$ is $\mathbb{H}$-linear and $f$ has an almost locally maximal horizontal directional derivative in the sense of Theorem \ref{almostmaximalityimpliesdifferentiability}. We will conclude that any Lipschitz function $f_{0}$ is Pansu differentiable at a point of $N$, proving Theorem \ref{maintheorem}. Our argument follows very closely that of \cite{DM11}, modified to use horizontal directions, $\mathbb{H}$-linear maps and H\"{o}lder equivalence of the Carnot-Carath\'{e}odory and Euclidean distance.

Recall the measure zero $G_{\delta}$ set $N$ and the notation $D^f$ fixed in Notation \ref{D^f}. In particular, the statement $(x,E)\in D^f$ implies that $x\in N$. Note that if $f-f_{0}$ is $\mathbb{H}$-linear then $D^{f}=D^{f_{0}}$ and also the functions $f$ and $f_{0}$ have the same points of Pansu differentiability.

\begin{proposition}\label{DoreMaleva}
Suppose $f_0:\mathbb{H}^n\to \mathbb{R}$ is a Lipschitz function, $(x_0,E_0)\in D^{f_0}$ and $\delta_0, \mu, K>0$. Then there is a Lipschitz function $f:\mathbb{H}^n\to \mathbb{R}$ such that $f-f_0$ is $\mathbb{H}$-linear with $\mathrm{Lip}_{\mathbb{H}}(f-f_{0})\leq \mu$, and a pair $(x_{\ast},E_{\ast})\in D^{f}$ with $d(x_{\ast},x_0)\leq \delta_0$ such that $E_{\ast}f(x_{\ast})>0$ is almost locally maximal in the following sense.

For any $\varepsilon>0$ there is $\delta_{\varepsilon}>0$ such that whenever $(x,E)\in D^{f}$ satisfies both:
\begin{enumerate}
\item $d(x,x_{\ast})\leq \delta_{\varepsilon}$, $Ef(x)\geq E_{\ast}f(x_{\ast})$,
\item for any $t\in (-1,1)$:
\begin{align*}
&|(f(x+tE_{\ast}(x))-f(x))-(f(x_{\ast}+tE_{\ast}(x_{\ast}))-f(x_{\ast}))|\\
& \qquad \leq K|t| ( Ef(x)-E_{\ast}f(x_{\ast}))^{\frac{1}{4}},
\end{align*}
\end{enumerate}
then:
\[Ef(x)<E_{\ast}f(x_{\ast})+\varepsilon.\]
\end{proposition}

We use the remainder of this section to prove Proposition \ref{DoreMaleva}. Fix parameters $f_{0}, x_0, E_0, \delta_0, \mu, K$ as given in the statement of the theorem.

\begin{assumptions}
Without loss of generality, we make the following assumptions:
\begin{itemize}
\item $K\geq 8$, since increasing $K$ makes the statement of Proposition \ref{DoreMaleva} stronger,
\item $\mathrm{Lip}_{\mathbb{H}}(f_0)\leq 1/2$, after multiplying $f_0$ by a positive constant and possibly increasing $K$,
\item $E_0f(x_0)\geq 0$, by replacing $E_0$ by $-E_0$ if necessary.
\end{itemize}
\end{assumptions}

We prove Proposition \ref{DoreMaleva} by using Algorithm \ref{alg} below to construct pairs $(x_{n}, E_{n})$ and Lipschitz functions $f_{n}$, satisfying various constraints, such that $E_{n}f(x_{n})$ is closer and closer to maximal. We then show that the limits $(x_{\ast},E_{\ast})$ and $f$ have the properties stated in Proposition \ref{DoreMaleva}. Algorithm \ref{alg} is an adaptation of \cite[Algorithm 3.2]{DM11}. We use the following notation to repeatedly find better pairs.

\begin{notation}\label{comparison}
Suppose $h:\mathbb{H}^n\to\mathbb{R}$ is Lipschitz, the pairs $(x,E)$ and $(x',E')$ belong to $D^h$, and $\sigma \geq 0$. We write:
\[(x,E)\leq_{(h,\sigma)} (x',E')\]
if $E h(x)\leq E' h(x')$ and for all $t\in (-1,1)$:
\begin{align*}
&|(h(x'+tE(x'))-h(x'))-(h(x+tE(x))-h(x))|\\
&\qquad \leq K (\sigma+ (E'f(x')-Ef(x))^{\frac{1}{4}})|t|.
\end{align*}
\end{notation}

In the language of Notation \ref{comparison}, Proposition \ref{DoreMaleva}(2) means $(x_{\ast},E_{\ast})\leq_{(f,0)} (x,E)$. 

In Algorithm \ref{alg} we introduce parameters satisfying various estimates, but the most important factor is the order in which the parameters are chosen. We use the following constants:
\begin{itemize}
\item $C_{\mathrm{a}}=C_{\mathrm{angle}}(1)\geq 1$ chosen by applying Lemma \ref{closedirectioncloseposition} with $S=1$,
\item $C_{\mathrm{H}}=C_{\mathrm{H\ddot{o}lder}}\geq 1$ denotes the constant in Proposition \ref{euclideanheisenberg} for the compact set $\overline{B_{\mathbb{H}}(x_{0},2+\delta_{0})} \subset \mathbb{H}^{n}$,
\item $C_{V}\geq 1$ such that $\mathrm{Lip}_{\mathbb{E}}(E)\leq C_{V}$ whenever $E\in V$ and $\omega(E)=1$. This is possible since $V=\mathrm{Span}\{X_{i}, Y_{i}: 1\leq i\leq n\}$ and $\{X_{i}, Y_{i}: 1\leq i\leq n\}$ are Lipschitz functions $\mathbb{R}^{2n+1} \to \mathbb{R}^{2n+1}$ with respect to the Euclidean distance.
\end{itemize}

Since $N$ is $G_{\delta}$ we can fix open sets $U_k\subset \mathbb{H}^n$ such that $N=\cap_{k=0}^{\infty} U_k$. We may assume that $U_{0}=\mathbb{H}^{n}$. 

\begin{algorithm}\label{alg}
Recall $f_0, x_0, E_0$ and $\delta_0$ from the hypotheses of Proposition \ref{DoreMaleva}. Let $\sigma_0:=2$ and $t_0:=\min \{1/4,\, \mu/2\}$.

Suppose that $m\geq 1$ and the parameters $f_{m-1}, x_{m-1}, E_{m-1}, \sigma_{m-1}, t_{m-1}, \delta_{m-1}$ have already been defined. Then we can choose:
\begin{enumerate}
\item $f_m(x):=f_{m-1}(x)+t_{m-1} \langle x, E_{m-1}(0) \rangle$,
\item $\sigma_m\in (0, \sigma_{m-1}/4)$,
\item $t_m\in (0, \min\{t_{m-1}/2,\, \sigma_{m-1}/(4m)\})$,
\item $\lambda_m\in (0, t_m\sigma_m^4/(2C_{\mathrm{a}}^4))$,
\item $D_m$ to be the set of pairs $(x,E)\in D^{f_m}=D^{f_0}$ such that $d(x,x_{m-1})<\delta_{m-1}$ and
\[(x_{m-1}, E_{m-1})\leq_{(f_m,\sigma_{m-1}-\varepsilon)} (x,E)\]
for some $\varepsilon\in (0,\sigma_{m-1})$,
\item $(x_m,E_m)\in D_m$ such that $Ef_m(x)\leq E_mf_m(x_m)+\lambda_m$ for every pair $(x,E)\in D_m$,
\item $\varepsilon_m\in (0,\sigma_{m-1})$ such that $(x_{m-1}, E_{m-1})\leq_{(f_m,\sigma_{m-1}-\varepsilon_m)} (x_m, E_m)$,
\item $\delta_m\in (0, (\delta_{m-1}-d(x_m,x_{m-1}))/2)$ such that $\overline{B_{\mathbb{H}}(x_m,\delta_m)}\subset U_m$ and for all $|t|<C_{\mathrm{H}}^{2}(1+C_{V})^{\frac{1}{2}}\delta_m^{\frac{1}{2}}/\varepsilon_m$:
\begin{align*}
&|(f_m(x_m + tE_{m}(x_{m}))-f_m(x_m))-(f_m(x_{m-1}+tE_{m-1}(x_{m-1}))-f_m(x_{m-1}))| \\
&\qquad \leq( E_mf_m(x_m)-E_{m-1}f_m(x_{m-1})+\sigma_{m-1}) |t|.
\end{align*}
\end{enumerate}
\end{algorithm}

\begin{proof}
Clearly one can make choices satisfying (1)--(5).

For (6) first notice that $(x_{m-1},E_{m-1})\in D_m$  and hence $D_{m} \neq \emptyset$. By Lemma \ref{lemmascalarlip}, the functions $f_{m} \colon \mathbb{H}^{n} \to \mathbb{R}$ are Lipschitz and 
\begin{equation}\label{deffm}
f_m(x)=f_0(x)+\Big\langle x, \sum_{k=0}^{m-1} t_kE_k(0)\Big\rangle.
\end{equation} 
Using $\mathrm{Lip}_{\mathbb{H}}(f_0)\leq 1/2$, $t_{k+1}\leq t_k/2$, $t_0\leq 1/4$ and Lemma \ref{lemmascalarlip} gives $\mathrm{Lip}_{\mathbb{H}}(f_m)\leq 1$. Lemma \ref{lipismaximal} implies $|Ef_{m}(x)|\leq \omega(E)\mathrm{Lip}_{\mathbb{H}}(f_{m})$. Hence $\sup_{(x,E)\in D_m} Ef_m(x)\leq 1$, so we can choose $(x_m,E_m)\in D_m$ as in (6). 

The definition of $D_m$ in (5) implies that one can choose $\varepsilon_m$ as in (7). 

Notice that (6) and the definition of $D_m$ in (5) imply that $x_{m}\in N\subset U_{m}$, $d(x_m,x_{m-1})<\delta_{m-1}$ and $E_mf_m(x_m)\geq E_{m-1}f_m(x_{m-1})$. Therefore we can use existence of the directional derivatives of $f_{m}$ to choose $\delta_m$ as in (8).
\end{proof}

We record for later use that $\mathrm{Lip}_{\mathbb{H}}(f_m)\leq 1$ for all $m\geq 1$. We next show that several parameters in Algorithm \ref{alg} converge to $0$ and the balls $B_{\mathbb{H}}(x_m,\delta_m)$ form a decreasing sequence.

\begin{lemma} \label{inclusionballs}
The sequences $\sigma_m, t_m, \lambda_m, \delta_m, \varepsilon_m$ converge to $0$. For every $m\geq 1$ the following inclusion holds:
\[\overline{B_{\mathbb{H}}(x_m,\delta_m)}\subset B_{\mathbb{H}}(x_{m-1},\delta_{m-1}).\]
\end{lemma}

\begin{proof}
Algorithm \ref{alg}(2) and $\sigma_0= 2$ gives $0 < \sigma_m \leq 2/4^m$ so $\sigma_m\to 0$. Combining this with Algorithm \ref{alg}(3,4,7,8) shows the other sequences converge to $0$. Let $x\in \overline{B_{\mathbb{H}}(x_m,\delta_m)}$. Then Algorithm \ref{alg}(6,8) gives:
\begin{align*}
d(x, x_{m-1})&\leq \delta_m+ d(x_m, x_{m-1})\\
&< \delta_{m-1}/2 + d(x_m,x_{m-1})/2\\
&<\delta_{m-1}.
\end{align*}
This shows the desired inclusion.
\end{proof}

Define $\varepsilon'_m>0$ by:
\begin{equation}\label{defepsprimo}
\varepsilon'_m:=\min\{\varepsilon_m/2,\, \sigma_{m-1}/2\}.
\end{equation}

We next show that the sets $D_{m}$ of special pairs form a decreasing sequence. This is an adaptation of \cite[Lemma 3.3]{DM11}.

\begin{lemma}\label{lemmachiave}
The following statements hold:
\begin{enumerate}
\item If $m\geq 1$ and $(x,E)\in D_{m+1}$ then:
\[(x_{m-1}, E_{m-1})\leq_{(f_m, \sigma_m-\varepsilon'_m)} (x,E),\]
\item If $m\geq 1$ then $D_{m+1}\subset D_m$,
\item If $m\geq 0$ and $(x,E)\in D_{m+1}$ then $ d(E(0),E_m(0))\leq \sigma_m$.
\end{enumerate}
\end{lemma}

\begin{proof}
If $m=0$ then (3) holds since:
\[d(E(0),E_0(0))\leq d(E(0),0)+d(0,E_0(0))\leq 2=\sigma_0.\]
It is enough to check that whenever $m\geq 1$ and (3) holds for $m-1$, then (1), (2) and (3) hold for $m$. Fix $m\geq 1$ and assume that (3) holds for $m-1$: 
\[ d(E(0),E_{m-1}(0)) \leq \sigma_{m-1}\quad  \mbox{for all}\quad (x, E)\in D_m.\]

\medskip

\emph{Proof of (1).} Algorithm \ref{alg}(6) states that $(x_m,E_m)\in D_m$. Hence:
\begin{equation}\label{stimavett}
d(E_m(0),E_{m-1}(0))\leq \sigma_{m-1}.
\end{equation}
Let $(x, E)\in D_{m+1}$. In particular, we have $Ef_{m+1}(x)\geq E_{m}f_{m+1}(x_{m})$ by Algorithm \ref{alg}(5). Notice that, since $\omega(E_{m})=\omega(E)=1$, we have $\langle E_m(0),E_m(0)\rangle=1$ and $\langle E(0), E_m(0) \rangle\leq 1$. Let $A:=Ef_m(x)-E_mf_m(x_m)$. Using Lemma \ref{lemmascalarlip} and the inequality $Ef_{m+1}(x)\geq E_{m}f_{m+1}(x_{m})$ gives:
\[A=Ef_{m+1}(x)-E_mf_{m+1}(x_m)-t_m\langle E(0), E_m(0) \rangle+t_m\geq 0.\]
Using Algorithm \ref{alg}(5) again, together with the above inequality, gives:
\[Ef_m(x)\geq E_mf_m(x_m)\geq E_{m-1} f_m(x_{m-1}).\]
In particular, $Ef_{m}(x)\geq E_{m-1}f_{m}(x_{m-1})$ proves the first part of the statement $(x_{m-1}, E_{m-1})\leq_{(f_m, \sigma_m-\varepsilon'_m)} (x,E)$.

Let $B:=E f_m(x)-E_{m-1}f_m(x_{m-1})\geq 0$. Lemma \ref{lipismaximal} and $\mathrm{Lip}_{\mathbb{H}}(f_m)\leq 1$ imply that $0\leq A, \, B\leq 2$. Using these inequalities and $K\geq 8$ gives:
\begin{align}\label{factorize}
K(B^{\frac{1}{4}}-A^{\frac{1}{4}})&\geq (B^{\frac{3}{4}}+B^{\frac{1}{2}}A^{\frac{1}{4}}+B^{\frac{1}{4}}A^{\frac{1}{2}}+A^{\frac{3}{4}})(B^{\frac{1}{4}}-A^{\frac{1}{4}})\nonumber  \\
&=B-A\nonumber \\
&=E_mf_m(x_m)-E_{m-1}f_m(x_{m-1}).
\end{align}
Since $A\geq Ef_{m+1}(x)-E_mf_{m+1}(x_m)$, \eqref{factorize} implies:
\begin{align}\label{estimate3}
&  E_mf_m(x_m)-E_{m-1}f_m(x_{m-1})+K( Ef_{m+1}(x)-E_mf_{m+1}(x_m))^{\frac{1}{4}}\nonumber \\
& \qquad \leq K B^{\frac{1}{4}}.
\end{align}
To prove the second part of $(x_{m-1}, E_{m-1})\leq_{(f_m, \sigma_m-\varepsilon'_m)} (x,E)$ we need to estimate:
\begin{equation}\label{thingtoestimate}
|(f_m(x + tE_{m-1}(x))-f_m(x))-(f_m(x_{m-1}+tE_{m-1}(x_{m-1}))-f_m(x_{m-1}))|.
\end{equation}
We consider two cases, depending on whether $t$ is small or large. 

\medskip

\emph{Suppose $|t|<C_{\mathrm{H}}^{2}(1+C_{V})^{\frac{1}{2}}\delta_m^{\frac{1}{2}}/\varepsilon_m$.}
Estimate \eqref{thingtoestimate} as follows:
\begin{align}\label{estimate}
&|(f_m(x+tE_{m-1}(x)) - f_m(x))-(f_m(x_{m-1}+tE_{m-1}(x_{m-1}))-f_m(x_{m-1}))|\nonumber \\
&\qquad \leq |(f_m(x + tE_m(x))-f_m(x)) - (f_m(x_m+tE_m(x_{m}))-f_m(x_m))| \nonumber \\
&\qquad \quad + |(f_m(x_m+ tE_m(x_{m}))-f_m(x_m))\nonumber \\
&\qquad \quad \qquad -(f_m(x_{m-1}+tE_{m-1}(x_{m-1}))-f_m(x_{m-1}))|  \nonumber \\
&\qquad \quad +| f_m(x+tE_{m-1}(x))-f_m(x+tE_m(x))|.
\end{align}
We consider the three terms on the right side of \eqref{estimate} separately.

Firstly, Algorithm \ref{alg}(1) and Lemma \ref{lemmascalarlip} give:
\begin{align}\label{eqz1}
&(f_m(x+tE_m(x))-f_m(x))-(f_m (x_{m}+tE_{m}(x_{m}))- f_m(x_{m}))\\
&\qquad =(f_{m+1}(x+tE_m(x))-f_{m+1}(x)) - ( f_{m+1}(x_m+tE_m(x_{m})) - f_{m+1}(x_m)) \nonumber \\
&\qquad \quad  -t_m\langle x+tE_m(x), E_m(0)\rangle +t_m\langle x,E_m(0)\rangle \nonumber \\
&\qquad \quad+t_m\langle x_m+tE_m(x_{m}), E_m(0)\rangle- t_m\langle x_m, E_m(0)\rangle \nonumber \\
&\qquad=(f_{m+1}(x+tE_m(x))-f_{m+1}(x)) - ( f_{m+1}(x_m+tE_m(x_{m})) - f_{m+1}(x_m))\nonumber .
\end{align}
Using \eqref{eqz1} and $(x,E)\in D_{m+1}$ then gives:
\begin{align}\label{estimate2}
&|(f_m(x+tE_{m}(x))-f_m(x))-(f_m(x_{m}+tE_{m}(x_{m}))-f_m(x_{m}))| \nonumber \\
&\qquad \leq K(\sigma_m+(Ef_{m+1}(x)-E_mf_{m+1}(x_m))^{\frac{1}{4}})|t|.
\end{align}

For the second term in \eqref{estimate} we recall that, for the values of $t$ we are considering, Algorithm \ref{alg}(8) states:
\begin{align*}
&|(f_m(x_m+tE_m(x_{m}))-f_m(x_m))-(f_m(x_{m-1}+tE_{m-1}(x_{m-1}))-f_m(x_{m-1}))| \\
&\qquad \leq( E_mf_m(x_m)-E_{m-1}f_m(x_{m-1})+\sigma_{m-1}) |t|.
\end{align*}

The final term in \eqref{estimate} is estimated using $\mathrm{Lip}_{\mathbb{H}}(f_{m})\leq 1$, \eqref{linestolines} and \eqref{stimavett}:
\begin{align*}
|f_m(x+tE_{m-1}(x))-f_m(x+tE_m(x))| &\leq d(x+tE_{m-1}(x), x+tE_m(x))\\
&= d(tE_{m-1}(0),tE_{m}(0)) \\
& \leq \sigma_{m-1}|t|.
\end{align*}

Adding the three estimates and using \eqref{estimate3} then \eqref{defepsprimo} and Algorithm \ref{alg}(2) allows us to develop \eqref{estimate}:
\begin{align*}
&|(f_m(x+ tE_{m-1}(x))-f_m(x))-(f_m(x_{m-1}+tE_{m-1}(x_{m-1}))-f_m(x_{m-1}))| \\
& \qquad \leq K(\sigma_m + (Ef_{m+1}(x)-E_mf_{m+1}(x_m))^{\frac{1}{4}})|t| \\
&\qquad \quad + ( E_mf_m(x_m)-E_{m-1}f_m(x_{m-1})+\sigma_{m-1}) |t|\\
&\qquad \quad + \sigma_{m-1}|t|\\
& \qquad \leq K(\sigma_{m-1}-\varepsilon_{m}' + (Ef_{m}(x)-E_{m-1}f_{m}(x_{m-1}))^{\frac{1}{4}})|t|.
\end{align*}
This gives the required estimate of \eqref{thingtoestimate} for small $t$.

\emph{Suppose $C_{\mathrm{H}}^{2}(1+C_{V})^{\frac{1}{2}}\delta_m^{\frac{1}{2}}/\varepsilon_m \leq |t| < 1$.}
In particular, this implies:
\begin{equation}\label{refest} \delta_{m} \leq \varepsilon_{m}^{2}t^{2}/C_{H}^{4}(1+C_{V}) \leq \varepsilon_{m}|t|.\end{equation}
The last inequality above used that
\[\varepsilon_{m}|t|/C_{H}^{4}(1+C_{V})\leq \varepsilon_{m}/C_{H}^{4}(1+C_{V})\leq 1,\]
using $\varepsilon_{m}\leq 2$ and $C_{H}, C_{V}\geq 1$.

Estimate \eqref{thingtoestimate} as follows:
\begin{align*}
&|(f_m(x+tE_{m-1}(x))-f_m(x))-(f_m(x_{m-1}+tE_{m-1}(x_{m-1}))-f_m(x_{m-1}))|\\
&\qquad \leq |(f_m(x_m+tE_{m-1}(x_{m}))-f_m(x_m))\\
&\qquad \quad \quad -(f_m(x_{m-1}+tE_{m-1}(x_{m-1}))-f_m(x_{m-1}))|\\
&\qquad \quad + |f_{m}(x) - f_{m}(x_{m})|\\
&\qquad \quad + |f_m(x+tE_{m-1}(x)) - f_m(x_m+tE_{m-1}(x_{m}))|.
\end{align*}

The estimate of the first term is given by Algorithm \ref{alg}(7). This states:
\[(x_{m-1}, E_{m-1})\leq_{(f_m, \sigma_{m-1}-\varepsilon_m)} (x_m, E_m),\] 
which gives the inequality:
\begin{align*}
&|(f_m(x_m+tE_{m-1}(x_{m}))-f_m(x_m))-(f_m(x_{m-1}+tE_{m-1}(x_{m-1}))-f_m(x_{m-1}))|\\
&\qquad \leq K(\sigma_{m-1}-\varepsilon_{m} + (E_{m}f_{m}(x_{m})-E_{m-1}f_{m}(x_{m-1}))^{\frac{1}{4}})|t|.
\end{align*}

The estimate of the second term uses $\mathrm{Lip}_{\mathbb{H}}(f_{m})\leq 1$ and \eqref{refest}:
\[|f_m(x)-f_m(x_m)|\leq d(x,x_{m}) \leq \delta_m\leq \varepsilon_m |t|\leq K\varepsilon_m |t|/8.\]

We estimate the final term using Proposition \ref{euclideanheisenberg} to compare the Carnot-Carath\'{e}odory and the Euclidean distance. Notice that $x+tE_{m-1}(x)$ and  $x_{m} + tE_{m-1}(x_{m})$ belong to $\overline{B_{\mathbb{H}}(x_{0},2+\delta_{0})}$. Hence we can use the constants $C_{\mathrm{H}}$ and $C_{V}$ defined immediately before Algorithm \ref{alg}. For our current values of $t$ we can estimate as follows:
\begin{align*}
&|f_m(x+tE_{m-1}(x))-f_m(x_m+tE_{m-1}(x_{m}))|\\
&\qquad \leq d(x+tE_{m-1}(x),x_{m}+tE_{m-1}(x_{m}))\\
&\qquad \leq C_{\mathrm{H}}|x+tE_{m-1}(x)-x_{m}-tE_{m-1}(x_{m})|^{\frac{1}{2}}\\
&\qquad \leq C_{\mathrm{H}}(1+C_{V})^{\frac{1}{2}}|x-x_{m}|^{\frac{1}{2}}\\
&\qquad \leq C_{\mathrm{H}}^2(1+C_{V})^{\frac{1}{2}}d(x,x_{m})^{\frac{1}{2}}\\
&\qquad \leq C_{\mathrm{H}}^{2}(1+C_{V})^{\frac{1}{2}}\delta_{m}^{\frac{1}{2}}\\
&\qquad \leq \varepsilon_{m}|t|\\
&\qquad \leq K\varepsilon_{m} |t|/8.
\end{align*}

Combine the estimates of the three terms and use $A\geq 0$ to obtain:
\begin{align*}
&|(f_m(x+tE_{m-1}(x))-f_m(x))-(f_m(x_{m-1}+tE_{m-1}(x_{m-1}))-f_m(x_{m-1}))|\\
&\qquad \leq K(\sigma_{m-1}-\varepsilon_m/2+(E_mf_m(x_m)-E_{m-1}f_m(x_{m-1}))^{\frac{1}{4}})\\
&\qquad \leq K(\sigma_{m-1}-\varepsilon'_m+(Ef_m(x)-E_{m-1}f_m(x_{m-1}))^{\frac{1}{4}}).
\end{align*}
This gives the correct estimate of \eqref{thingtoestimate} for large $t$. Combining the two cases proves (1) for $m$.

\medskip

\emph{Proof of (2).} Suppose $(x,E)\in D_{m+1}$. Then $(x,E)\in D^{f_{m+1}}=D^{f_{m}}$ and Lemma \ref{inclusionballs} implies $d(x,x_{m-1})<\delta_{m-1}$. Combining this with (1) gives $(x,E)\in D_{m}$. This proves (2) for $m$.

\medskip
 
\emph{Proof of (3).} Suppose $(x,E)\in D_{m+1}$. Then $E_m f_{m+1}(x_m)\leq E f_{m+1}(x)$ using Algorithm \ref{alg}(5). Equivalently, by Algorithm \ref{alg}(1):
\[E_m f_{m}(x_m)+t_m\langle E_{m}(0),E_{m}(0) \rangle \leq E f_{m}(x)+t_m\langle E(0),E_{m}(0) \rangle.\]
Also $(x,E)\in D_m$ by (2) above, so Algorithm \ref{alg}(6) implies:
\[Ef_m(x)\leq E_m f_m(x_m)+\lambda_m.\]
Combining the two inequalities above gives $t_m\leq t_m\langle E(0),E_m(0)\rangle+\lambda_m$. Rearranging, this implies:
\[\langle E(0),E_{m}(0) \rangle\geq 1-\lambda_m/t_m.\]
Lemma \ref{closedirectioncloseposition} applied to $g(t):=tE(0)$ and $h(t):=tE_{m}(0)$ gives:
\begin{align*}
d(E(0),E_m(0)) &\leq C_{\mathrm{a}}| E(0)- E_m(0)|^{\frac{1}{2}}\\
&=C_{\mathrm{a}}(2-2\langle E(0), E_m(0) \rangle)^{\frac{1}{4}}\\
& \leq C_{\mathrm{a}}(2\lambda_m/t_m)^{\frac{1}{4}}\\
& \leq \sigma_m
\end{align*}
by Algorithm \ref{alg}(4). This proves (3) for $m$.
\end{proof}

We next study the convergence of $(x_{m}, E_{m})$ and $f_{m}$. We show that the directional derivatives converge to a directional derivative of the limiting function, and the limit of $(x_{m},E_{m})$ belongs to $D_{m}$ for every $m$. This is our adaptation of \cite[Lemma 3.4]{DM11}.

\begin{lemma}\label{lemmaquasifinale}
The following statements hold:
\begin{enumerate}
\item $f_{m}\to f$ pointwise, where $f:\mathbb{H}^n\to \mathbb{R}$ is Lipschitz and $\mathrm{Lip}_{\mathbb{H}}(f)\leq 1$,
\item $f-f_m$ is $\mathbb{H}$-linear and $\mathrm{Lip}_{\mathbb{H}}(f-f_m) \leq 2t_m$ for $m\geq 0$,
\item $x_{m} \to x_{\ast}\in N$ and $E_m(0) \to E_{\ast}(0)$ for some $E_{\ast} \in V$ with $\omega(E_{\ast})=1$. For $m \geq 0$:
\[d(x_{\ast},x_m)< \delta_m \quad \text{ and }\quad d(E_{\ast}(0),E_m(0))\leq \sigma_m,\]
\item $E_{\ast}f(x_{\ast})$ exists, is strictly positive and $E_mf_m(x_m)\uparrow E_{\ast}f(x_{\ast})$,
\item $(x_{m-1},E_{m-1})\leq_{(f_m, \sigma_{m-1}-\varepsilon'_m)} (x_{\ast},E_{\ast})$ for $m\geq 1$,
\item $(x_{\ast},E_{\ast})\in D_m$ for $m\geq 1$.
\end{enumerate}
\end{lemma}

\begin{proof}
We prove each statement individually.

\medskip

\emph{Proof of (1).} Algorithm \ref{alg}(1) gives $f_m(x)=f_0(x)+\langle x,\sum_{k=0}^{m-1}t_kE_{k}(0)\rangle$. Define $f:\mathbb{H}^n\to\mathbb{R}$ by
\begin{equation}\label{deff}
f(x):=f_0(x)+\Big\langle x,\sum_{k=0}^{\infty}t_k E_{k}(0)\Big\rangle.
\end{equation}
Notice $|f(x)-f_m(x)|\leq | x | \sum_{k=m}^{\infty} t_k |E_{k}(0)|$. Hence $f_m\to f$ pointwise and, since $\mathrm{Lip}_{\mathbb{H}}(f_m)\leq 1$, we deduce $\mathrm{Lip}_{\mathbb{H}}(f)\leq 1$.

\medskip

\emph{Proof of (2).} Lemma \ref{lemmascalarlip} shows that $f-f_{m}$ is $\mathbb{H}$-linear. Using also Algorithm \ref{alg}(3) shows that for every $m\geq 0$:
\[\mathrm{Lip}_{\mathbb{H}}(f-f_m)  \leq \sum_{k=m}^{\infty} t_k \leq t_m \sum_{k=m}^{\infty} \frac{1}{2^{k-m}} \leq 2t_m.\]

\medskip

\emph{Proof of (3).} Let $q\geq m\geq 0$. The definition of $D_{q+1}$ in Algorithm \ref{alg}(5) shows that $(x_q,E_q)\in D_{q+1}$. Hence Lemma \ref{lemmachiave}(2,3) imply that $(x_q,E_q)\in D_{m+1}$, and consequently:
\begin{equation}\label{Cauchy1}
d(E_q(0),E_m(0)) \leq \sigma_m.
\end{equation}
Since $(x_q,E_q)\in D_{m+1}$, Algorithm \ref{alg}(5) implies:
\begin{equation}\label{Cauchy2}
d(x_q,x_m)< \delta_m.
\end{equation}

Since $\sigma_m, \delta_m \to 0$ we see that $(x_m)_{m=1}^{\infty}$ and $(E_{m}(0))_{m=1}^{\infty}$ are Cauchy sequences, so converge to some $x_{\ast}\in\mathbb{H}^{n}$ and $v\in \mathbb{H}^{n}$. Since $E_{m}\in V$ and $\omega(E_{m})=1$, we know $|p(v)|=1$ and $v_{2n+1}=0$. Using group translations, we can extend $v$ to a vector field $E_{\ast}\in V$ with $\omega(E_{\ast})=1$ and $E_{\ast}(0)=v$. Letting $q\to \infty$ in \eqref{Cauchy1} and \eqref{Cauchy2} implies $d(E_{\ast}(0),E_m(0)) \leq \sigma_m$ and $d(x_{\ast},x_m)\leq \delta_m$. Lemma \ref{inclusionballs} then gives the strict inequality $d(x_{\ast},x_m)< \delta_m$.

We now know that $x_{\ast}\in \overline{B_{\mathbb{H}}(x_m,\delta_m)}$ for every $m\geq 1$. Recall that $N=\cap_{m=0}^{\infty} U_m$ for open sets $U_m \subset \mathbb{H}^{n}$, and Algorithm \ref{alg}(8) states $\overline{B_{\mathbb{H}}(x_{m},\delta_{m})}\subset U_{m}$. Hence $x_{\ast}\in N$.

\medskip

\emph{Proof of (4).} As in the proof of (3) we have $(x_q,E_q)\in D_{m+1}$ for every $q\geq m\geq 0$. Therefore, using Lemma \ref{lemmachiave}(1), $q\geq m\geq 1$ implies:
\begin{equation}\label{bla}
(x_{m-1}, E_{m-1})\leq_{(f_m,\sigma_{m-1}-\varepsilon'_m)} (x_q,E_q).
\end{equation}
Applying Algorithm \ref{alg}(1) and \eqref{bla} (with $m$ and $q$ replaced by $q+1$) gives:
\begin{equation}\label{stima} 
E_qf_q(x_q)< E_q f_{q+1}(x_q)\leq E_{q+1}f_{q+1}(x_{q+1}) \quad \mbox{for}\quad q \geq 0.
\end{equation} 
Hence $(E_q f_q(x_q))_{q=0}^{\infty}$ is strictly increasing and positive as $E_0f_0(x_0)\geq 0$.

Recall $\mathrm{Lip}_{\mathbb{H}}(f_q)\leq 1$ for $q\geq 1$. Hence, by Lemma \ref{lipismaximal}, the sequence $(E_qf_q(x_q))_{q=1}^{\infty}$ is bounded above by $1$. Consequently $E_qf_q(x_q)\to L$ for some $0<L\leq 1$. Inequality \eqref{stima} implies that also $E_q f_{q+1}(x_q) \to L$. Further:
\[E_q f(x_q)=E_q f_q(x_q)+E_q (f-f_q)(x_q)\]
and $|E_q (f-f_q)(x_q)|\leq \mathrm{Lip}_{\mathbb{H}}(f-f_q) \leq 2t_{q} \to 0$. Hence also $E_qf(x_q) \to L$.

Let $q\geq m\geq 0$ and consider:
\[s_{m,q}:=E_qf_m(x_q)-E_{m-1}f_m(x_{m-1}).\]
By \eqref{bla} we have $s_{m,q}\geq 0$. Letting $q\to \infty$, writing $f_{m}=f+(f_{m}-f)$, and using $\mathbb{H}$-linearity of $f_{m}-f$ and $E_qf(x_q)\to L$  implies:
\begin{equation}\label{defiC}
s_{m,q}\to s_{m}:=(f_m-f)(E_{\ast}(0))+L-E_{m-1}f_m(x_{m-1})\geq 0.
\end{equation}
Also $s_{m} \to 0$ as $m\to \infty$ since $\mathrm{Lip}_{\mathbb{H}}(f_{m}-f)\leq 2t_{m}$ and $E_{m-1}f_{m}(x_{m-1})\to L$.
Using \eqref{bla} shows that for $t\in (-1,1)$:
\begin{align}\label{bla2}
&|(f_m(x_q+ tE_{m-1}(x_{q})) - f_m(x_q))-(f_m(x_{m-1} + tE_{m-1}(x_{m-1}))-f_m(x_{m-1}))|\nonumber \\
&\qquad \leq K (\sigma_{m-1}-\varepsilon'_m+(s_{m,q})^{\frac{1}{4}}) |t|.
\end{align}
Letting $q\to \infty$ in \eqref{bla2} shows that for $t\in (-1,1)$:
\begin{align}\label{eqncruc}
&|(f_m(x_{\ast} + tE_{m-1}(x_{\ast}))-f_m(x_{\ast}))-(f_m(x_{m-1}+tE_{m-1}(x_{m-1}))-f_m(x_{m-1}))|\nonumber \\
&\qquad \leq K(\sigma_{m-1}-\varepsilon'_m+(s_{m})^{\frac{1}{4}})|t|.
\end{align}
Using $\mathrm{Lip}_{\mathbb{H}}(f)\leq 1$ and $d(E_{\ast}(0),E_{m-1}(0))\leq \sigma_{m-1}$ from (3) of the present Lemma gives:
\begin{align*}
|f(x_{\ast}+tE_{\ast}(x_{\ast}))-f(x_{\ast}+tE_{m-1}(x_{\ast}))|&\leq d(x_{\ast}(tE_{\ast}(0)),x_{\ast}(tE_{m-1}(0)))\\
&\leq \sigma_{m-1}|t|.
\end{align*}
Since $f-f_{m}$ is $\mathbb{H}$-linear we can use $\mathrm{Lip}_{\mathbb{H}}(f-f_m) \leq 2t_m$ to estimate:
\begin{align*}
|(f-f_m)(x_{\ast}+tE_{m-1}(x_{\ast}))-(f-f_m)(x_{\ast})| &= |(f-f_{m})(tE_{m-1}(0))|\\
&\leq t\mathrm{Lip}_{\mathbb{H}}(f-f_{m})\\
&\leq 2t_{m}|t|.
\end{align*}

Combining the previous three inequalities shows that for $t\in (-1,1)$:
\begin{align*}
&|(f(x_{\ast}+tE_{\ast}(x_{\ast}))-f(x_{\ast}))-(f_m(x_{m-1}+tE_{m-1}(x_{m-1}))-f_m(x_{m-1}))|\\
&\qquad \leq |(f_m(x_{\ast}+tE_{m-1}(x_{\ast}))-f_m(x_{\ast}))\\
&\qquad \quad \quad -(f_m(x_{m-1}+tE_{m-1}(x_{m-1}))-f_m(x_{m-1}))|\\
&\qquad \quad +|f(x_{\ast}+tE_{\ast}(x_{\ast}))-f(x_{\ast}+tE_{m-1}(x_{\ast}))|\\
&\qquad \quad +|(f-f_m)(x_{\ast}+tE_{m-1}(x_{\ast}))-(f-f_m)(x_{\ast})|\\
&\qquad \leq (K(\sigma_{m-1}-\varepsilon'_m+(s_{m})^{\frac{1}{4}})+\sigma_{m-1}+2t_m)|t|.
\end{align*}

Fix $\varepsilon>0$ and choose $m\geq 1$ such that:
\[K(\sigma_{m-1}-\varepsilon'_m+(s_{m})^{\frac{1}{4}})+\sigma_{m-1}+2t_m\leq \varepsilon/3\]
and
\[|E_{m-1}f_m(x_{m-1})-L|\leq \varepsilon/3.\]
Using the definition of $E_{m-1}f_{m}(x_{m-1})$, fix $0<\delta<1$ such that for $|t|< \delta$:
\[|f_m(x_{m-1}+tE_{m-1}(x_{m-1}))-f_m(x_{m-1})-tE_{m-1}f_m(x_{m-1})|\leq \varepsilon|t|/3.\]
Hence for $|t|< \delta$:
\begin{align*}
&|f(x_{\ast}+tE_{\ast}(x_{\ast}))-f(x_{\ast})-tL|\\
&\qquad \leq|(f(x_{\ast}+tE_{\ast}(x_{\ast}))-f(x_{\ast}))-(f_m(x_{m-1}+tE_{m-1}(x_{m-1}))-f_m(x_{m-1}))|\\
&\qquad \quad +|f_m(x_{m-1}+tE_{m-1}(x_{m-1}))-f_m(x_{m-1})-tE_{m-1}f_m(x_{m-1})|\\
&\qquad \quad +|E_{m-1}f_m(x_{m-1})-L| |t|\\
&\qquad \leq \varepsilon |t|.
\end{align*}
This proves that $E_{\ast}f(x_{\ast})$ exists and is equal to $L$. We already saw that $(E_qf_q(x_q))_{q=1}^{\infty}$ is a strictly increasing sequence of positive numbers. This proves (4).

\medskip

\emph{Proof of (5).} The definition of $L$ and Lemma \ref{lemmascalarlip} implies:
\[E_{\ast}f_m(x_{\ast})=L+E_{\ast}(f_{m}-f)(x_{\ast})=L+(f_m-f)(E_{\ast}(0)).\]
Using \eqref{defiC} shows $s_{m}=E_{\ast}f_m(x_{\ast})-E_{m-1}f_m(x_{m-1})\geq 0$. Substituting this in \eqref{eqncruc} gives (5).

\medskip

\emph{Proof of (6).} Property (6) is a consequence of (3), (4) and (5).
\end{proof}

We now prove that the limit directional derivative $E_{\ast}f(x_{\ast})$ is almost locally maximal in horizontal directions. This is our adaptation of \cite[Lemma 3.5]{DM11}. 

\begin{lemma}\label{almostlocmax}
For all $\varepsilon>0$ there is $\delta_{\varepsilon}>0$ such that if $(x,E)\in D^f$ satisfies $d(x_{\ast},x)\leq \delta_{\varepsilon}$ and $(x_{\ast},E_{\ast})\leq_{(f,0)}(x,E)$ then:
\[Ef(x)<E_{\ast}f(x_{\ast})+\varepsilon.\]
\end{lemma}

\begin{proof}
Fix $\varepsilon>0$. Use Lemma \ref{inclusionballs} to choose $m\geq 1$ such that:
\begin{equation}\label{param} 
m\geq 4/\varepsilon^{\frac{3}{4}}\quad \mbox{and}\quad \lambda_m,t_m\leq \varepsilon/4.
\end{equation}
Recall $\varepsilon'_m=\min\{\varepsilon_m/2,\, \sigma_{m-1}/2\}$. Using Lemma \ref{lemmaquasifinale}(3), fix $\delta_{\varepsilon}>0$ such that $\delta_{\varepsilon}<\delta_{m-1}-d(x_{\ast},x_{m-1})$ such that for every 
$|t|< C_{\mathrm{H}}^{2}(1+C_{V})^{\frac{1}{2}}\delta_{\varepsilon}^{\frac{1}{2}}/\varepsilon_{m}'$:
\begin{align}\label{estimated2}
&|(f_m(x_{\ast}+tE_{\ast}(x_{\ast}))-f_m(x_{\ast}))-(f_m(x_{m-1}+tE_{m-1}(x_{m-1}))-f_m(x_{m-1}))|\nonumber \\
&\qquad \leq (E_{\ast}f_m(x_{\ast})-E_{m-1}f_m(x_{m-1})+\sigma_{m-1})|t|.
\end{align}
Such $\delta_{\varepsilon}$ exists since Lemma \ref{lemmaquasifinale}(5) implies $E_{\ast}f_m(x_{\ast})\geq E_{m-1}f_m(x_{m-1})$. 

We argue by contradiction. Suppose that $(x,E)\in D^f$ satisfies $d(x_{\ast},x)\leq \delta_{\varepsilon}$, $(x_{\ast},E_{\ast})\leq_{(f,0)} (x,E)$ and $Ef(x)\geq E_{\ast}f(x_{\ast})+\varepsilon$. We plan to show that $(x,E)\in D_m$. We first observe that this gives a contradiction. Indeed, Algorithm \ref{alg}(6) and the monotone convergence $E_mf_m(x_m)\uparrow E_{\ast}f(x_{\ast})$ would then imply:
\[Ef_m(x)\leq E_mf_m(x_m)+\lambda_m\leq E_{\ast}f(x_{\ast})+\lambda_m.\]
From Lemma \ref{lemmaquasifinale}(2) and \eqref{param} we deduce:
\begin{align*}
Ef(x)-E_{\ast}f(x_{\ast})&=(Ef_m(x)-E_{\ast}f(x_{\ast}))+E(f-f_m)(x)\\
&\leq \lambda_m+2t_m\\
&\leq 3\varepsilon /4.
\end{align*}
This contradicts the assumption that $Ef(x)\geq E_{\ast}f(x_{\ast})+\varepsilon$.

\medskip

\emph{Proof that $(x,E)\in D_m$.} Notice that $(x,E)\in D^{f_{m}}$ since $(x,E)\in D^f$ and $D^f=D^{f_{m}}$ because $f-f_{m}$ is $\mathbb{H}$-linear. Next observe:
\[d(x,x_{m-1})\leq d(x,x_{\ast})+d(x_{\ast},x_{m-1}) < \delta_{m-1}.\]
Hence it suffices to show that $(x_{m-1},E_{m-1})\leq_{(f_m, \sigma_{m-1}-\varepsilon'_m/2)} (x,E)$.
Lemma \ref{lipismaximal} implies:
\[|E(f-f_{m})(x)|,\, |E_{\ast}(f-f_{m})(x_{\ast})|\leq \mathrm{Lip}_{\mathbb{H}}(f-f_{m}).\]
Hence using \eqref{param} and our definition of $(x,E)$, we deduce:
\begin{align*}
Ef_m(x)-E_{\ast}f_m(x_{\ast})&\geq Ef(x)-E_{\ast}f(x_{\ast})-2\mathrm{Lip}_{\mathbb{H}}(f_m-f)\\
&\geq \varepsilon -4t_m\geq 0.
\end{align*}
Lemma \ref{lemmaquasifinale}(6) states that $(x_{\ast}, E_{\ast})\in D_{m}$, which implies $E_{m-1}f_m(x_{m-1})\leq E_{\ast}f_m(x_{\ast})$. Hence:
\[Ef_m(x)\geq E_{\ast}f_m(x_{\ast})\geq E_{m-1}f_m(x_{m-1}).\]
In particular, the inequality $Ef_{m}(x)\geq E_{m-1}f_m(x_{m-1})$ proves the first half of the statement $(x_{m-1},E_{m-1})\leq_{(f_m, \sigma_{m-1}-\varepsilon'_m/2)} (x,E)$.

We next deduce several inequalities from our hypotheses. Denote:
\begin{itemize}
\item $A:=Ef(x)-E_{\ast}f(x_{\ast})$,
\item $B:=Ef_m(x)-E_{\ast}f_m(x_{\ast})$,
\item $C:=Ef_m(x)-E_{m-1}f_m(x_{m-1})$.
\end{itemize}
Our definition of $(x,E)$ states $A\geq \varepsilon$, while the inequalities above give $0\leq B\leq C$. Also $A,\, B,\, C\leq 2$ by Lemma \ref{lipismaximal}. Recall the factorization:
\begin{equation}\label{factorizerepeat}
(A^{\frac{1}{4}}-B^{\frac{1}{4}})(B^{\frac{3}{4}}+B^{\frac{1}{2}}A^{\frac{1}{4}}+B^{\frac{1}{4}}A^{\frac{1}{2}}+A^{\frac{3}{4}})=A-B.
\end{equation}
Using Lemma \ref{lemmaquasifinale}(2), \eqref{param} and Algorithm \ref{alg}(3), we obtain:
\begin{align*}
A^{\frac{1}{4}}-B^{\frac{1}{4}} &\leq (A-B)/\varepsilon^{\frac{3}{4}}\\
&=(E(f-f_m)(x)-E_{\ast}(f-f_m)(x_{\ast}))/\varepsilon^{\frac{3}{4}}\\
&\leq 4t_m /\varepsilon^{\frac{3}{4}}\\
& \leq mt_m\\
&\leq \sigma_{m-1}/4.
\end{align*}
Since $B,\, C\leq 2$ and $K\geq 8$ we have:
\[B^{\frac{3}{4}}+B^{\frac{1}{2}}A^{\frac{1}{4}}+B^{\frac{1}{4}}A^{\frac{1}{2}}+A^{\frac{3}{4}}\leq 8 \leq K.\]
Hence using \eqref{factorizerepeat} with $A$ replaced by $C$ gives:
\[KC^{\frac{1}{4}}-KB^{\frac{1}{4}}\geq C-B=E_{\ast}f_m(x_{\ast})-E_{m-1}f_m(x_{m-1}).\]
Combining our estimates gives:
\begin{align}\label{stima32}
&E_{\ast}f_m(x_{\ast})-E_{m-1}f_m(x_{m-1})+K(Ef(x)-E_{\ast}f(x_{\ast}))^{\frac{1}{4}}\nonumber \\
&\qquad =E_{\ast}f_m(x_{\ast})-E_{m-1}f_m(x_{m-1})+KA^{\frac{1}{4}}\nonumber\\
&\qquad \leq KC^{\frac{1}{4}}-KB^{\frac{1}{4}}+K(B^{\frac{1}{4}}+\sigma_{m-1}/4) \nonumber\\
&\qquad = K((Ef_m(x)-E_{m-1}f_m(x_{m-1}))^{\frac{1}{4}}+\sigma_{m-1}/4).
\end{align}

We can now prove the second half of $(x_{m-1},E_{m-1})\leq_{(f_m, \sigma_{m-1}-\varepsilon'_m/2)} (x,E)$. We need to estimate:
\begin{equation}\label{incases}
|(f_m(x+tE_{m-1}(x))-f_m(x))-(f_m(x_{m-1}+tE_{m-1}(x_{m-1}))-f_m(x_{m-1}))|.
\end{equation}
We consider two cases, depending on whether $t$ is small or large.

\medskip

\emph{Suppose $|t|\leq C_{\mathrm{H}}^{2}(1+C_{V})^{\frac{1}{2}}\delta_{\varepsilon}^{\frac{1}{2}}/\varepsilon_{m}'$.}
To estimate \eqref{incases} we use the inequality:
\begin{align}\label{toestimate}
&|(f_m(x+tE_{m-1}(x)) - f_m(x))-(f_m(x_{m-1}+tE_{m-1}(x_{m-1})) - f_m(x_{m-1})|\nonumber \\
&\qquad \leq |(f_m(x + tE_{\ast}(x)) - f_m(x))-(f_m(x_{\ast}+tE_{\ast}(x_{\ast})) - f_m(x_{\ast}))| \nonumber \\
&\qquad \quad + |(f_m(x_{\ast}+tE_{\ast}(x_{\ast})) - f_m(x_{\ast}))\nonumber \\
&\qquad \quad \qquad -(f_m(x_{m-1}+tE_{m-1}(x_{m-1}))-f_m(x_{m-1}))| \nonumber \\
&\qquad \quad +|f_m(x+tE_{m-1}(x))-f_m(x+tE_{\ast}(x))|.
\end{align}
Using Lemma \ref{lemmascalarlip}, the hypothesis $(x_{\ast},E_{\ast})\leq_{(f,0)} (x,E)$ and $\mathbb{H}$-linearity of $f_{m}-f$, we can estimate the first term in \eqref{toestimate}:
\begin{align}\label{estimated1}
&|(f_m(x+tE_{\ast}(x_{\ast}))-f_m(x))-(f_m(x_{\ast}+tE_{\ast}(x_{\ast}))-f_m(x_{\ast}))|\nonumber \\
&\qquad =|f(x+tE_{\ast}(x))-f(x))-(f(x_{\ast}+tE_{\ast}(x_{\ast}))-f(x_{\ast}))| \nonumber \\
&\qquad \leq K(Ef(x)-E_{\ast}f(x_{\ast}))^{\frac{1}{4}}|t|.
\end{align}
Using \eqref{estimated2} and the assumption that $t$ is small bounds the second term in \eqref{toestimate} by:
\[(E_{\ast}f_m(x_{\ast})-E_{m-1}f_m(x_{m-1})+\sigma_{m-1})|t|.\]
Lemma \ref{lemmaquasifinale} implies that the third term of \eqref{toestimate} is bounded above by $\sigma_{m-1}|t|$.
By combining the estimates of each term and using \eqref{stima32} we develop \eqref{toestimate}:
\begin{align}\label{hi}
&|(f_m(x+tE_{m-1}(x))-f_m(x)) - (f_m(x_{m-1}+tE_{m-1}(x_{m-1}))-f_m(x_{m-1}))| \nonumber \\
&\qquad \leq (K(Ef(x)-E_{\ast}f(x_{\ast}))^{\frac{1}{4}} + E_{\ast}f_m(x_{\ast})-E_{m-1}f_m(x_{m-1})+2\sigma_{m-1})|t| \nonumber \\
&\qquad \leq (K((Ef_{m}(x)-E_{m-1}f_{m}(x_{m-1}))^{\frac{1}{4}}+\sigma_{m-1}/4) + 2\sigma_{m-1})|t|\nonumber \\
&\qquad \leq K(\sigma_{m-1}-\varepsilon'_m/2+(Ef_m(x)-E_{m-1}f_m(x_{m-1}))^{\frac{1}{4}})|t|,
\end{align}
using $\varepsilon'_m\leq \sigma_{m-1}/2$ and $K\geq 8$ in the final line. This gives the correct estimate of \eqref{incases} for small $t$.

\medskip

\emph{Suppose $C_{\mathrm{H}}^{2}(1+C_{V})^{\frac{1}{2}}\delta_{\varepsilon}^{\frac{1}{2}}/\varepsilon_{m}' \leq |t|\leq 1$.}
To estimate \eqref{incases} we use the inequality:
\begin{align*}
&|(f_m(x+tE_{m-1}(x))-f_m(x))-(f_m(x_{m-1}+tE_{m-1}(x_{m-1}))-f_m(x_{m-1}))|\\
&\qquad \leq |(f_m(x_{\ast}+tE_{m-1}(x_{\ast}))-f_m(x_{\ast}))\\
&\qquad \quad \quad -(f_m(x_{m-1}+tE_{m-1}(x_{m-1}))-f_m(x_{m-1}))|\\
&\qquad \quad +|f_m(x_{\ast}) - f_m(x)| + |f_m(x+tE_{m-1}(x)) - f_m(x_{\ast}+tE_{m-1}(x_{\ast}))|.
\end{align*}
Lemma \ref{lemmaquasifinale}(5) shows that the first term is bounded above by:
\[K(\sigma_{m-1}-\varepsilon'_m+( E_{\ast}f_m(x_{\ast})-E_{m-1}f_m(x_{m-1}))^{\frac{1}{4}})|t|.\]
The second term is bounded by $d(x_{\ast},x)\leq \delta_{\varepsilon} \leq \varepsilon_{m}'|t|\leq K\varepsilon_{m}'|t|/4$. For the third term we use Proposition \ref{euclideanheisenberg} to relate the Carnot-Carath\'{e}odory distance and the Euclidean distance. Notice $x+tE_{m-1}(x), x_{\ast}+tE_{m-1}(x_{\ast}) \in \overline{B_{\mathbb{H}}(x_{0},2+\delta_{0})}$. Hence we can use the constants $C_{\mathrm{H}}$ and $C_{V}$ fixed before Algorithm \ref{alg}:
\begin{align*}
&|f_m(x+tE_{m-1}(x)) - f_m(x_{\ast}+tE_{m-1}(x_{\ast}))|\\
&\qquad \leq C_{\mathrm{H}}|x+tE_{m-1}(x)-x_{\ast}-tE_{m-1}(x_{\ast})|^{\frac{1}{2}}\\
&\qquad \leq C_{\mathrm{H}}(1+C_{V})^{\frac{1}{2}}|x_{\ast}-x|^{\frac{1}{2}}\\
&\qquad \leq C_{\mathrm{H}}^2(1+C_{V})^{\frac{1}{2}}d(x_{\ast},x)^{\frac{1}{2}}\\
&\qquad \leq C_{\mathrm{H}}^{2}(1+C_{V})^{\frac{1}{2}}\delta_{\varepsilon}^{\frac{1}{2}}\\
&\qquad \leq \varepsilon_{m}'|t|\\
&\qquad \leq K\varepsilon_{m}'|t|/4.
\end{align*}
Putting together the three estimates and using $E_{\ast}f_m(x_{\ast})\leq Ef_m(x)$ gives:
\begin{align*}
&|(f_m(x+tE_{m-1}(x))-f_m(x))-(f_m(x_{m-1}+tE_{m-1}(x_{m-1}))-f_m(x_{m-1}))| \\
&\qquad \leq K(\sigma_{m-1}-\varepsilon'_m/2+( Ef_m(x)-E_{m-1}f_m(x_{m-1}))^{\frac{1}{4}})|t|.
\end{align*}
This gives the correct estimate of \eqref{incases} for large $t$.

Combining the two cases estimates \eqref{incases} for any $t\in (-1,1)$. Hence:
\[(x_{m-1},E_{m-1})\leq_{(f_m,\sigma_{m-1}-\varepsilon'_m/2)}(x,E).\]
This concludes the proof.
\end{proof}

We can now conclude by proving Proposition \ref{DoreMaleva} and hence, using also Theorem \ref{almostmaximalityimpliesdifferentiability}, prove Theorem \ref{maintheorem}.

\begin{proof}[Proof of Proposition \ref{DoreMaleva}]
Lemma \ref{lemmaquasifinale} and Lemma \ref{almostlocmax} prove Proposition \ref{DoreMaleva}. Indeed, Lemma \ref{lemmaquasifinale} states that there is $f\colon \mathbb{H}^{n} \to \mathbb{R}$ Lipschitz such that $f-f_{0}$ is linear and $\mathrm{Lip}_{\mathbb{H}}(f-f_{0})\leq 2t_{0}\leq \mu$. It also states that there is $(x_{\ast}, E_{\ast})\in D^{f}$ satisfying, among other properties, $d(x_{\ast},x_{0})<\delta_{0}$ and  $E_{\ast}f(x_{\ast})>0$. Lemma \ref{almostlocmax} then shows that $E_{\ast}f(x_{\ast})$ is almost locally maximal in the sense of Proposition \ref{DoreMaleva}.
\end{proof}

\begin{proof}[Proof of Theorem \ref{maintheorem}]
Let $f_{0}\colon \mathbb{H}^{n}\to \mathbb{R}$ be a Lipschitz function. Multiplying $f_{0}$ by a non-zero constant does not change the set of points where it is Pansu differentiable. Hence we can assume $\mathrm{Lip}_{\mathbb{H}}(f_{0})\leq 1/4$. Fix an arbitrary pair $(x_{0},E_{0})\in D^{f_{0}}$. 

Apply Proposition \ref{DoreMaleva} with $\delta_{0}=1$, $\mu=1/4$ and $K=8$. This gives a Lipschitz function $f\colon \mathbb{H}^{n}\to \mathbb{R}$ such that $f-f_{0}$ is $\mathbb{H}$-linear with $\mathrm{Lip}_{\mathbb{H}}(f-f_{0})\leq 1/4$ and a pair $(x_{\ast},E_{\ast})\in D^{f}$, in particular $x_{\ast}\in N$, such that $E_{\ast}f(x_{\ast})>0$ is almost locally maximal in the following sense.

For any $\varepsilon>0$ there is $\delta_{\varepsilon}>0$ such that whenever $(x,E)\in D^{f}$ satisfies both:
\begin{enumerate}
\item $d(x,x_{\ast})\leq \delta_{\varepsilon}$, $Ef(x)\geq E_{\ast}f(x_{\ast})$,
\item for any $t\in (-1,1)$:
\begin{align*}
&|(f(x+tE_{\ast}(x))-f(x))-(f(x_{\ast}+tE_{\ast}(x_{\ast}))-f(x_{\ast}))|\\
& \qquad \leq 8|t| ( Ef(x)-E_{\ast}f(x_{\ast}) )^{\frac{1}{4}},
\end{align*}
\end{enumerate}
then:
\[Ef(x)<E_{\ast}f(x_{\ast})+\varepsilon.\]

Combining $\mathrm{Lip}_{\mathbb{H}}(f_{0})\leq 1/4$ and $\mathrm{Lip}_{\mathbb{H}}(f-f_{0})\leq 1/4$ gives $\mathrm{Lip}_{\mathbb{H}}(f)\leq 1/2$. Notice that $(x_{\ast},E_{\ast})$ is also almost locally maximal in the sense of Theorem \ref{almostmaximalityimpliesdifferentiability}, since the restriction on pairs above is weaker than that in Theorem \ref{almostmaximalityimpliesdifferentiability}. Hence Theorem \ref{almostmaximalityimpliesdifferentiability} implies that $f$ is Pansu differentiable at $x_{\ast}\in N$. A $\mathbb{H}$-linear function is Pansu differentiable everywhere. Consequently $f_{0}$ is Pansu differentiable at $x_{\ast}$, proving Theorem \ref{maintheorem}.
\end{proof}

\end{document}